\documentclass{article}
\usepackage[utf8]{inputenc}
\usepackage{bbm}

\usepackage[numbers]{natbib}

\RequirePackage{amsthm,amsmath,amsfonts,amssymb}
\RequirePackage{graphicx,url,color}

\usepackage{comment}
\usepackage{soul}
\usepackage{algpseudocode}
\usepackage{algorithm,color,xcolor}
\usepackage{enumitem,amsmath}   

\newtheorem{theorem}{Theorem}[section]
\newtheorem{lemma}[theorem]{Lemma}
\newtheorem{definition}[theorem]{Definition}
\newtheorem*{theorem*}{Theorem}
\newtheorem{example}[theorem]{Example}

\newtheorem{proposition}[theorem]{Proposition}
\newtheorem{remark}[theorem]{Remark}

\newcommand{\RR}{\mathbb{R}}
\newcommand{\R}{\mathbb{R}}

\newcommand{\Tn}{\mathcal{U}_m}

\newcommand{\Trop}{\text{Trop}}
\newcommand{\T}{\mathbb R^e \!/\mathbb R {\bf 1}}

\newcommand{\tconv}{\text{tconv\,}}

\begin{document}

\title{Hit and Run Sampling \\from Tropically Convex Sets}
\author{Ruriko Yoshida  \and  Keiji Miura \and David Barnhill}
\date{}

\maketitle

\begin{abstract}
    In this paper we propose Hit and Run (HAR) sampling from a tropically convex set. The key ingredient of HAR sampling from a tropically convex set is sampling uniformly from a tropical line segment over the tropical projective torus, which runs linearly in its computational time complexity.  We show that this HAR sampling method samples uniformly from a tropical polytope which is the smallest tropical convex set of finitely many vertices.  Finally, we apply this novel method to any given distribution using Metropolis-Hasting filtering over a tropical polytope. 
\end{abstract}

\section{Introduction}

Hit and Run (HAR) sampling is one of the most popular Markov Chain Monte Carlo (MCMC) methods used to sample random points from an arbitrary distribution over a closed convex set in an Euclidean space using a ``line.''

In 1971, Turcin introduced the basic structure for a HAR sampler to generate points in an Euclidean space~\citep{Turcin}.
Then, in 1979, Boneh and Golen developed an HAR sampler to sample uniformly from a compact convex set~\citep{Boneh}.
In 1984, Smith worked on geometric variations and convergence of a sample via HAR sampler \citep{smith} and then, in 1993, Belisle et al.~extended this to a general distribution~\citep{Belisle}.
In several publications, Lov\'asz and Vempalla studied convergence rates of a HAR sampler in high dimensional Euclidean space \citep{Lovasz1999,Lovasz2003,Lovasz2006}.
In 2018, Chen et al.~developed a fast HAR sampler to generate random points over a polytope in an Euclidean space \citep{fastHAR}.
Corte and Montiel developed a Matrix HAR (MHAR) algorithm to sample points from a polytope over an Euclidean space \citep{MHR}.
See \cite{Luengo} for more recent work in MCMC samplers. 

While use of HAR samplers over Euclidean space has been studied thoroughly, there has been little research related to HAR sampling in the tropical projective torus $\mathbb R^e \!/\mathbb R {\bf 1}$ which is isometric to $\R^{e-1}$.
In this research, we propose a HAR sampler with the tropical metric to sample random points from arbitrary distribution on a tropically convex set over a tropical projective torus in terms of the max-plus arithmetic.

The key ingredient of a HAR sampler is to sample a random point from a ``line'' over a given closed set.  In the tropical setting, we use a tropical line segment over a tropically convex set as a ``line'' in a HAR sampler.  Specifically, we use the fact that a tropical line segment is ``intrinsic,'' that is, the fact that since a tropical line segment between two points is tropically convex if end points of the tropical line segment are inside of a tropically convex set, all points in the tropical line segment are inside of the tropically convex set.  Therefore, if we know how to sample a random point inside of a tropical line segment, we can apply it to a HAR sampler over a tropically convex set. In this paper we develop a novel HAR sampler to sample random points from a {\em tropical polytope} which is a  tropically convex set of finitely many vertices over the tropical projective torus.  Our method iteratively runs a Markov chain by the following steps: (1) compute tropical convex hulls of random subsets of vertices; (2) project the current point in the given tropical polytope onto these tropical convex hulls computed in the previous step (extrapolation); (3) sample a point uniformly from a tropical line segment between the projections computed at Step 2; (4) set the point randomly sampled as the starting point for the next iteration; and then (5) iterate step (1) through step (4) until it converges.  We call our MCMC sampler the {\em vertex HAR using extrapolation}. Our main result in this paper is 
\begin{theorem*}
The vertex HAR using extrapolation described in Algorithm \ref{alg:HAR_extrapolation} and Algorithm \ref{alg:HAR_extrapolation2} samples random points uniformly from a given tropical polytope.  
\end{theorem*}
Then we show that we can sample from any given distribution over a tropical polytope using our HAR sampler with the tropical metric combined with the Metropolis-Hasting filtering.  In addition, we discuss how to apply our HAR sampler to a {\em space of phylogenetic trees} with a given set of leaf labels, which is a tropically convex set.  We end this paper with a discussion on how our HAR sampler with the tropical metric can be applied to estimation of the volume of a tropical polytope over the tropical projective torus, phylogenomics, and extreme value statistics on causal inference.   


This paper is organized as follows: in Section~\ref{BASICS} we discuss basics of tropical arithmetic that are the building blocks of tropical HAR samplers.  Section~\ref{T_SAMPLE} discusses HAR sampling from a tropically convex set beginning with sampling from a tropical line segment and building up to sampling from a tropical polytope.  Section~\ref{EXP} illustrates computational experiments of tropical HAR samplers introduced in Section~\ref{T_SAMPLE} on tropical polytopes.  Finally, in Section~\ref{sec:ultra}, we apply HAR samplers to a space of ultrametrics related to phylogenetic trees.  

We conduct computational experiments with {\tt R}, statistical computational tool.  Our {\tt R} code used for this paper is available upon a request to the first author.  

\section{Tropical Basics}\label{BASICS}

Throughout this paper, we consider the tropical projective torus $\mathbb R^e \!/\mathbb R {\bf 1}$ which is isomorphic to $\R^{e-1}$.
For more details, see \cite{ETC} and \cite{MS}.

\begin{definition}[Tropical Arithmetic Operations]
Under the tropical semiring $(\,\mathbb{R} \cup \{-\infty\},\oplus,\odot)\,$, we have the tropical arithmetic operations of addition and multiplication defined as:
$$x \oplus y := \max\{x, y\}, ~~~~ x \odot y := x + y ~~~~\mbox{  where } x, y \in \mathbb{R}\cup\{-\infty\}.$$
Note that $-\infty$ is the identity element under addition $\oplus$ and $0$ is the identity element under multiplication $\odot$ over this semiring.
\end{definition}

\begin{definition}[Tropical Scalar Multiplication and Vector Addition]
For any $x,y \in \mathbb{R}\cup \{-\infty\}$ and for any $v = (v_1, \ldots ,v_e),\; w= (w_1, \ldots , w_e) \in (\mathbb{R}\cup-\{\infty\})^e$, we have tropical scalar multiplication and tropical vector addition defined as:
$$x \odot v \oplus y \odot w := (\max\{x+v_1,y+w_1\}, \ldots, \max\{x+v_e,y+w_e\}).$$
\end{definition}

\begin{definition}\label{def:polytope}
Suppose we have $S \subset \mathbb R^e \!/\mathbb R {\bf 1}$. If 
\[
x \odot v \oplus y \odot w \in S
\]
for any $x, y \in \R$ and for any $v, w \in S$, then $S$ is called {\em tropically convex}.
Suppose $V = \{v^1, \ldots , v^s\}\subset \mathbb R^e \!/\mathbb R {\bf   1}$.  The smallest tropically-convex subset containing $V$ is called the {\em tropical convex hull} or {\em tropical polytope} of $V$ which can be written as the set of all tropical linear combinations of $V$
$$ \mathrm{tconv}(V) = \{a_1 \odot v^1 \oplus a_2 \odot v^2 \oplus \cdots \oplus a_s \odot v^s \mid  a_1,\ldots,a_s \in \R \}.$$
A {\em tropical line segment} between two points $v^1, \, v^2$ is a tropical polytope, $\mathcal{P}$, of a set of two points $\{v^1, \, v^2\} \subset \mathbb R^e \!/\mathbb R {\bf   1}$.
\end{definition}

\begin{definition}[Generalized Hilbert Projective Metric]
\label{eq:tropmetric} 
For any points $v:=(v_1, \ldots , v_e), \, w := (w_1, \ldots , w_e) \in \mathbb R^e \!/\mathbb R {\bf 1}$,  the {\em tropical distance} (also known as {\em tropical metric}) $d_{\rm tr}$ between $v$ and $w$ is defined as:
\begin{equation*}
d_{\rm tr}(v,w)  := \max_{i \in \{1, \ldots , e\}} \bigl\{ v_i - w_i \bigr\} - \min_{i \in \{1, \ldots , e\}} \bigl\{ v_i - w_i \bigr\}.
\end{equation*}
\end{definition}
Next we remind the reader of the definition of a projection in terms of the tropical metric onto a tropical polytope.
The tropical projection formula can be found as Formula 5.2.3 in \cite{MS}. 
\begin{definition}\label{def:proj}
Let $V:= \{v^1, \ldots, v^s\} \subset \mathbb{R}^e/{\mathbb R} {\bf 1}$ and let $\mathcal P = \tconv(v^1, \ldots, v^s)\subseteq \mathbb R^{e}/\RR{\bf 1}$ be a tropical polytope with its vertex set $V$.
For $x \in \T$, let
\begin{equation}\label{eq:tropproj} 
\pi_\mathcal{P} (x) \!:=\! \bigoplus\limits_{l=1}^s \lambda_l \odot  v^l, ~ ~ {\rm where} ~ ~ \lambda_l \!=\! {\rm min}(x-v^l).
\end{equation}

Then 
\[
d_{\rm tr}(x, \pi_\mathcal{P} (x))  \leq d_{\rm tr}(x, y)
\]
for all $y \in \mathcal P$.  In other words, $\pi_\mathcal{P} (x)$ is the projection of $x \in \T$ in terms of the tropical metric $d_{\rm tr}$ onto the tropical polytope $\mathcal P$.
\end{definition}

We are interested in the region of ambient points in terms of $\pi_\mathcal{P}(x)$.
According to the projection rule, i.e.,  Equation \eqref{eq:tropproj}, two general nearby points 
are projected to the same position if they have the same $\lambda_l$ for all $l$.
This condition takes place at least when the minimum in ${\rm min}(x-v^l)$ in Equation \eqref{eq:tropproj} is attained at the same (say, $j$-th) coordinate for all $l$.
Thus, we consider the region of $x$ where $\lambda_l$ for all $l$ includes $x_j$ for fixed $j$, i.e., $\lambda_l=x_j -v^l_j$ for all $l$ so that all the points in that region have the same $\lambda_l$.
In fact, $\lambda$ becomes a constant as $\lambda_l = -w^l_j$ after $x_j {\bf   1}$ is subtracted under $\mathbb{R}^e/\mathbb R {\bf   1}$.
And, thus, $\pi_\mathcal{P}(x) = \lambda \cdot V$ for all $x$ in the region represents the same point.
This argument can be summarized as Lemma \ref{lm:project}.

\begin{figure}[t!]
 \centering
 \includegraphics[width=0.45\textwidth]{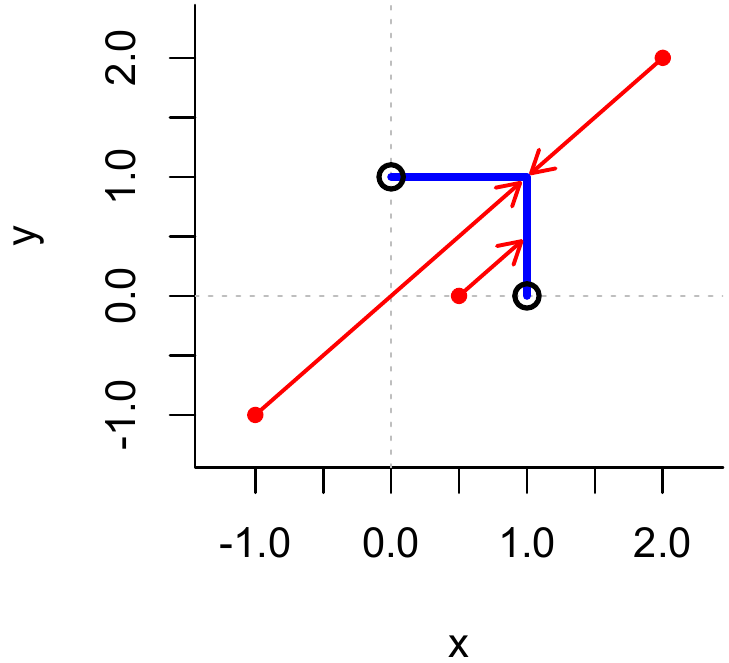} ~
 \includegraphics[width=0.45\textwidth]{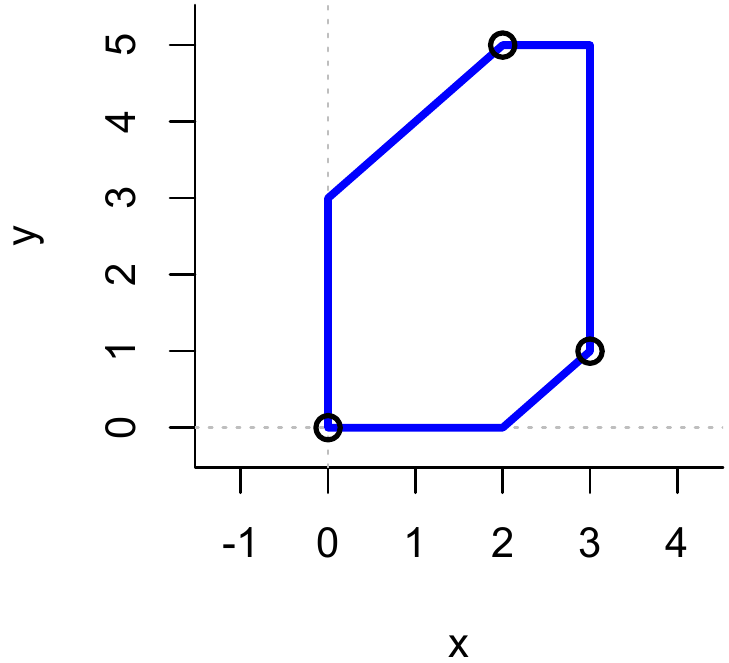}\newline\newline
 \includegraphics[width=0.45\textwidth]{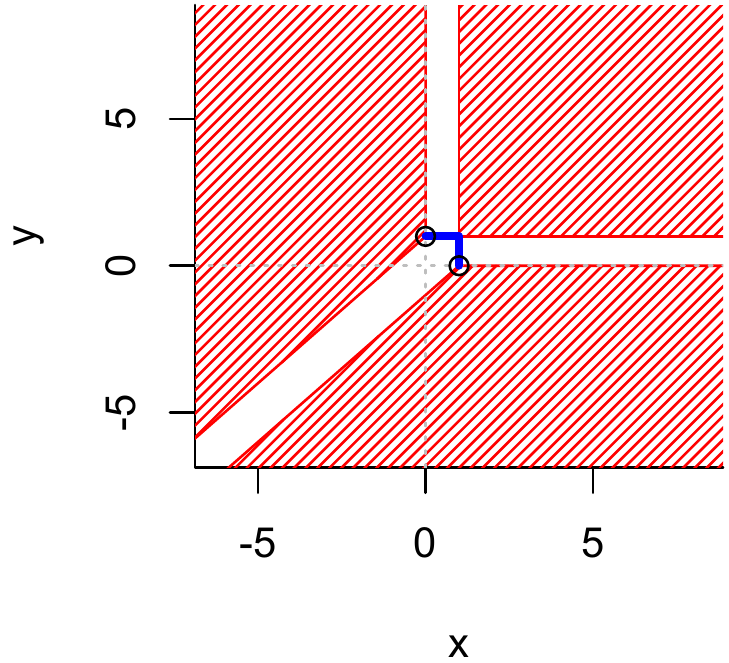} ~
 \includegraphics[width=0.45\textwidth]{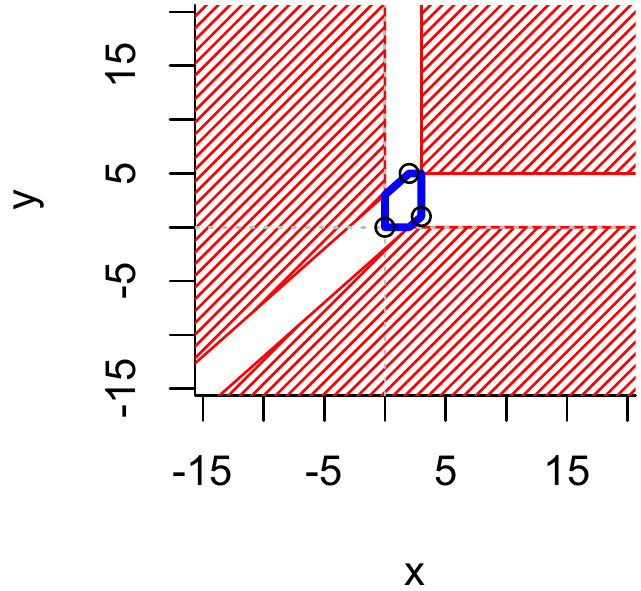}
\caption{(top left) Tropical polytope of two points $(0, 1, 0), \, (0, 0, 1)$ in $\mathbb R^3 \!/\mathbb R {\bf 1}$. (top right) Tropical polytope of three points $(0, 0, 0), \, (0, 3, 1), \, (0, 2, 5)$ in $\mathbb R^3 \!/\mathbb R {\bf 1}$. (bottom left) Three regions that are projected to $(0, 1, 0), \, (0, 0, 1), \, (0, 1, 1)$. The points in each of the three red colored regions are projected to the same point. Note the red regions are fairly wide and most of the points that are randomly sampled on the plane can be projected only to one of the three points. (bottom right) Three regions that are projected to $(0, 0, 3), \, (0, 2, 0), \, (0, 3, 5)$ of a unit polytope. The points in each of the three red colored regions are projected to the same point. Note the red regions are fairly wide and most of the points that are randomly sampled on the plane can be projected only to one of the three points.}
\label{fig:tropPoly}
\end{figure}

\begin{lemma}\label{lm:project}
Let $\mathcal P = \tconv(v^1, \ldots, v^s)\subseteq \mathbb R^{e}/\RR{\bf 1}$ be a tropical polytope with its vertex set $\{v^1, \ldots, v^s\} \subset \mathbb{R}^e/{\mathbb R} {\bf 1}$ where $v^l:=(v_1^l, \ldots, v_e^l)$ for $l = 1, \ldots , s$.
Let $x=(x_1, \ldots , x_e) \in \T$ such that $x_j \leq x_k + \min_{l= 1, \ldots , e}(v_j^l-v_k^l)$ for fixed $j$ and for all $k$.
Then $\pi_\mathcal{P}(x)_i = \max_l(v_i^l-v_j^l)$ with $\lambda_l=-v_j^l$.
That is, all the points $x$ satisfying the above inequalities are projected to the same point.
\end{lemma}
\begin{proof}
Let $x_j \leq x_k + \min_{l= 1, \ldots , e}(v_j^l-v_k^l)$ for all $k$.
Then $x_j \leq x_k + (v_j^l-v_k^l)$ for all $k$ and all $l$.
Or $x_j - v_j^l \leq x_k -v_k^l$ for all $k$ and all $l$.
Then $\lambda_l = {\rm min}(x-v^l) = x_j - v_j^l$ for all $l$.
\end{proof}

\begin{remark}
When $x$ is far away from the origin, i.e. in the limit of small $w$, the condition for the $j$-th region becomes $x_j \leq x_k$ for $j \neq k$.
That is, $j$-th region is the region where $j$-th coordinate $x_j$ is the minimum.
Union of the 1st, 2nd, \ldots, and $e$-th regions cover the entire space, because at least one of the coordinates must be the minimum.
This suggest that almost all the points are project to only $j$ points.
\end{remark}

\begin{example}
We consider the tropical polytope of two points $(0, 1, 0), \, (0, 0, 1)$ in $\mathbb R^3 \!/\mathbb R {\bf 1}$ in Figure \ref{fig:tropPoly} (Left). Note that this tropical line segment passes through $(0,1,1)$.

First, to gain an intuition for the projection, we consider some example points in Figure \ref{fig:tropPoly} (Top Left).
For $x=(0,2,2)$, we have $\pi_{\mathcal{P}}(x) = (0, 1, 1)$ with $\lambda_1=\lambda_2=0$.
For $x=(0,0,0)$, we have $\pi_{\mathcal{P}}(x) = (0, 1, 1)$ with $\lambda_1=\lambda_2=-1$.
For $x=(0,1/2,0)$, we have $\pi_{\mathcal{P}}(x) = (0, 0, 1/2)$ with $\lambda_1=1/2, \lambda_2=-1$.
Note that two different points that are projected to the same point can have different values of $\lambda$, although $\lambda_2-\lambda_1$ is unique.

Importantly, the general result in Lemma \ref{lm:project} demonstrates that there is a notable tendency that most points far from the origin are projected to the same point.
The points in the region $x_3 \leq x_1=0$ and $x_3\leq x_2 -1$ are projected to $(0,1,0)$, which corresponds the $j=3$ case where $x_3$ or $x_3 - 1$ is always the minimum when we determine $\lambda$.
For example, for $x=(0,100,-1000)$, we have $\pi_{\mathcal{P}}(x) = (0, 1, 0)$ with $\lambda_1=-1000, \lambda_2=-1001$.
The points in the region $0=x_1 \leq x_2 -1$ and $0=x_1\leq x_3 -1$ are projected to $(0,1,1)$, which corresponds the $j=1$ case where $x_1$ is always the minimum when we determine $\lambda$.
The points in the region $x_2 \leq x_1=0$ and $x_2\leq x_3 -1$ are projected to $(0,0,1)$, which corresponds the $j=2$ case where $x_2$ or $x_2 - 1$ is always the minimum when we determine $\lambda$.
Only the other points in the very narrow regions are projected to the points other than the above three points of the polytope.
\end{example}

\begin{example}
We consider the tropical polytope of three points, $(1, 0, 0)$, $(0, 1, 0)$, and $(0, 0, 1)$,
in $\mathbb R^3 \!/\mathbb R {\bf 1}$ in Figure \ref{fig:tropPoly} (Right).
By Lemma \ref{lm:project}, the points in each of the three red regions in Figure \ref{fig:tropPoly} (Bottom Right) are projected to $(0, 1, 1)$, $(0, 0, -1)$, $(0, -1, 0)$, respectively.
Note that only the points in the very narrow white region in the entire plane are projected to the other points of the polytope.
\end{example}

\section{Sampling from a Tropical Convex Hull}\label{T_SAMPLE}

\subsection{Sampling from a Tropical Line Segment}

From the proof of Proposition 5.2.5 in \cite{MS}, a tropical line segment with a given pair of vectors $u=:(u_1, \ldots , u_e), v := (v_1, \ldots , v_e) \in \T$  can be written as follows:  Without loss of generality, we assume that $(v_1 - u_1) \geq \ldots \geq (v_{e-1} - u_{e-1}) \geq (v_e - u_e) = 0$ after permuting coordinates of $v - u$.
Then the tropical line segment $\Gamma_{u, v}$ from $v$ to $u$ is
\begin{equation}\label{eq:troline}
\left\{
\begin{array}{ccl}
\!\!\!(v_e \!-\! u_e) \odot u \oplus v \!\!\!\! &=& \!\!\! v\\ 
\!\!\!(v_{e-1} \!-\! u_{e-1}) \odot u \oplus v \!\!\!\!&=&\!\!\! (v_1, v_2, v_3, \ldots , v_{e-1}, v_{e-1} - u_{e-1} + u_e)\\
&\vdots& \\
\!\!\!(v_2 \!-\! u_2) \odot u \oplus v \!\!\!\!&=&\!\!\! (v_1, v_2, v_2 - u_2 + u_3, \ldots , v_2 - u_2 + u_e)\\
\!\!\! (v_1 \!-\! u_1) \odot u \oplus v \!\!\!\!&=&\!\!\! u .\\
\end{array}\right.    
\end{equation}
That is, we can represent $u$ as
\begin{eqnarray}
u &=& v
+ (\Delta_{e-1}-\Delta_e) \left( \begin{array}{c} 0\\ \vdots\\ 0 \\ 0 \\ 1 \end{array} \right)
+ (\Delta_{e-2}-\Delta_{e-1}) \left( \begin{array}{c} 0\\ \vdots \\ 0 \\ 1 \\ 1 \end{array} \right) \nonumber \\
&+& \ldots 
+(\Delta_{2}-\Delta_3) \left( \begin{array}{c} 0\\ 0\\ 1\\ \vdots\\ 1 \end{array} \right)
+(\Delta_{1}-\Delta_2) \left( \begin{array}{c} 0\\ 1\\ 1\\ \vdots\\ 1 \end{array} \right)
\label{eq:representation}
\end{eqnarray}
where $\Delta_i := v_i-u_i$.

\begin{example}
For $v=(2,0,8)$ and $u=(-1,-1,2)$, we have $v-u=(3,1,6)$ so $\Delta = (6,3,1)$ after permutation.  After similarly permuting $v$ and $u$ where $v=(8,2,0)$ and $u=(2,-1,-1)$ we have $u = v + 2(0,0,1) + 3(0,1,1)$.
\end{example}

Using this fact we have the following algorithm to sample a random point from $\Gamma_{u, v}$.
\begin{algorithm}
\caption{Sampling from Tropical Line Segment}\label{alg:troplineseg}
\begin{algorithmic}
\State {\bf Input:} $u, v \in \mathbb R^e \!/\mathbb R {\bf 1}$
\State {\bf Output:} A random point $x \in \Gamma_{u, v}$.
\State Sample $\ell$ uniformly from $[\min(v-u), \max(v-u)]$.
\State Set $x:= \ell \odot u \oplus v = (\max(\ell+u_1, v_1), \ldots , \max(\ell+u_e, v_e))$. \\
\Return $x$.
\end{algorithmic}
\end{algorithm}

\begin{proposition}
We can sample a point uniformly from $\Gamma_{u, v}$ shown in \eqref{eq:troline} via Algorithm \ref{alg:troplineseg}. Sampling a random point via Algorithm \ref{alg:troplineseg} is $O(e)$.
\end{proposition}

\begin{lemma}\label{lm:tropln}
Suppose $u, v \in \T$ are the input of Algorithm \ref{alg:troplineseg} and $x(\ell_1)$ and $x(\ell_2) \in \T$ are outputs from Algorithm \ref{alg:troplineseg} for $\ell_1$ and $\ell_2$, respectively.
Then 
\[
d_{\rm tr}(x(\ell_1), x(\ell_2)) = |\ell_2 - \ell_1|.
\]
\end{lemma}
\begin{proof}
Without loss of generality, we assume that $0 = (v_e - u_e) \leq (v_{e-1} - u_{e-1}) \leq \ldots \leq (v_1 - u_1)$ after permuting coordinates of $v - u$.
Similar to Equation \eqref{eq:representation}, $x$ for $\ell$ can be represented as
\begin{eqnarray}
x(\ell) &=& v
+ (\Delta_{e-1}-\Delta_e) \left( \begin{array}{c} 0\\ \vdots\\ 0 \\ 0 \\ 1 \end{array} \right)
+ (\Delta_{e-2}-\Delta_{e-1}) \left( \begin{array}{c} 0\\ \vdots \\ 0 \\ 1 \\ 1 \end{array} \right) \nonumber \\
&+& \ldots 
+(\Delta_{j}-\Delta_{j+1}) \left( \begin{array}{c} 0\\ \vdots\\ 0\\ 0\\ 1\\ \vdots\\ 1 \end{array} \right)
+(\ell-\Delta_j) \left( \begin{array}{c} 0\\ \vdots\\ 0\\ 1\\ 1\\ \vdots\\ 1 \end{array} \right) , \nonumber
\end{eqnarray}
where $\ell \in (v_j - u_j, v_{j-1} - u_{j-1}]$ and there are $j-1$ zeros in the last column vector.
Suppose $\ell_1 < \ell_2$ with $\ell_1 \in (v_{j_1} - u_{j_1}, v_{j_1-1} - u_{j_1-1}]$ and $\ell_2 \in (v_{j_2} - u_{j_2}, v_{j_2-1} - u_{j_2-1}]$ for some $j_1 \geq j_2$.
Then we have
\begin{equation}
x(\ell_2) - x(\ell_1) = (0, \ldots, 0, \ell_2-\Delta_{j_2}, \ldots, \ell_2 - \Delta_{j_1-1}, \ell_2 - \ell_1, \ldots, \ell_2 - \ell_1),\nonumber
\end{equation}
whose max and min are $\ell_2 - \ell_1$ and $0$.
\end{proof}

\begin{proposition}\label{prop:tropln2}
Suppose $u, v \in \T$ are the input of Algorithm \ref{alg:troplineseg} and $x \in \T$ is an output from Algorithm \ref{alg:troplineseg}.  Then 
\[
d_{\rm tr}(v, x) = \ell - \min(v - u).
\]

\end{proposition}
\begin{proof}
Use Lemma \ref{lm:tropln} for $x(\ell_1)=v$, i.e., $\ell_1=\min(v - u)$.

\end{proof}

\begin{lemma}\label{lm:1to1}
For fixed $u, v \in \mathbb R^e \!/\mathbb R {\bf 1}$, let a map $G_{u, v} : [\min(v-u), \max(v-u)] \to \mathbb R^e \!/\mathbb R {\bf 1}$ such that
\[
G_{u, v} (\ell) =  \ell \odot u \oplus v = (\max(\ell+u_1, v_1), \ldots , \max(\ell+u_e, v_e)).
\]
Then $G_{u, v}$ is a one-to-one map.  
\end{lemma}
\begin{proof}
Suppose $\ell_1, \, \ell_2 \in [\min(v-u), \max(v-u)]$ such that $\ell_1 \not = \ell_2$.  Then $G_{u, v}(\ell_1) = \ell_1 \odot u \oplus v, \, G_{u, v}(\ell_2) =\ell_2 \odot u \oplus v \in \Gamma_{u, v}$ by Definition \ref{def:polytope} and by Proposition \ref{prop:tropln2}  we have
\[
d_{\rm tr}(v, G_{u, v}(\ell_1)) = \ell_1 - \min(v - u),
\]
and
\[
d_{\rm tr}(v, G_{u, v}(\ell_2)) = \ell_2 - \min(v - u).
\]
Therefore 
\[
d_{\rm tr}(v, G_{u, v}(\ell_1)) \not = d_{\rm tr}(v, G_{u, v}(\ell_2)).
\]
Thus, since $d_{\rm tr}$ is a metric, $d_{\rm tr}(x, y) = 0$ if and only if $x = y$, and since $G_{u, v}(\ell_1), \, G_{u, v}(\ell_2) \in \Gamma_{u, v}$, 
\[
G_{u, v}(\ell_1) \not = G_{u, v}(\ell_2).
\]
\end{proof}

\begin{lemma}\label{lm:prob}
For fixed $u, v \in \mathbb R^e \!/\mathbb R {\bf 1}$, the probability to sample a point $x \in \Gamma_{u, v} \subset \mathbb R^e \!/\mathbb R {\bf 1}$ via Algorithm \ref{alg:troplineseg} is uniformly distributed with the probability density function 
\[
f(x) = \frac{1}{d_{tr}(u,v)}.
\]
\end{lemma}
\begin{proof}
Suppose $x \in \Gamma_{u, v}$ is sampled via Algorithm \ref{alg:troplineseg}.  From Algorithm \ref{alg:troplineseg}, $\ell \in \mathbb{R}$ is sampled uniformly from $[\min(v-u), \max(v-u)]$.  By Lemma \ref{lm:1to1}, the map 
\[
G_{u, v} (\ell) =  \ell \odot u \oplus v = (\max(\ell+u_1, v_1), \ldots , \max(\ell+u_e, v_e))
\]
in Algorithm \ref{alg:troplineseg} is a one-to-one map.  Thus, $G_{u, v(\ell)}$ is sampled uniformly with the probability density function 
\[
f(x) = \frac{1}{[\max(v-u)-\min(v-u)]}=\frac{1}{d_{tr}(u,v)}.
\]
\end{proof}


We now show that any line segment, $\Gamma_{u,v}$, may be extended without adding any break points to form a new line segment $\Gamma_{u',v'}$ such that $\Gamma_{u,v} \subset \Gamma_{u',v'}$. We will use this idea in developing Algorithm~\ref{alg:HAR_vert3} in Subsection~\ref{sec:vert_HAR_samp}.

\begin{theorem}\label{thm:ext}
Given a tropical line segment $\Gamma_{u,v} \in \mathbb R^e \!/\mathbb R {\bf 1}$ with endpoints $u$ and $v$ and no more than $e-2$ break points, $\mathbf{b}$, the line segment may be extended from the endpoints without increasing the number of break points to form a new line segment $\Gamma'_{u^{'},v^{'}}$ with new end points $u^{'}$ and $v^{'}$ if:
\begin{enumerate}
\item The proposed new endpoints $u^{'}$ and $v^{'}$ exhibit a change in the same coordinates of the vector $d_i=i-b_i$ which shows the coordinate difference between endpoint $i\in\{u,v\}$ and the break point immediately prior to the original endpoints, defined as $b_i$.
\item For each new endpoint, $i'$, each coordinate of the difference vector $d_{i'}=i'-b_i$ is zero or some value $\delta$.  That is $d_{i'}^j\in\{0,\delta\}$.
\end{enumerate}
\end{theorem}

\begin{proof}
To extend a tropical line segment, $\Gamma_{u,v}$, without increasing the number of break points, $|\mathbf{b}|$, we only need consider an endpoint, $u$, and the break point immediately prior to $u$, which we call $b_u \in \mathbf{b}$. In cases where $|\mathbf{b}|=0$, $b_{u}=v$, the other endpoint.

Critically, a line segment $\Gamma_{u,v}$ where $|\mathbf{b}|>0$, is made up of a series of tropical line segments between each $b\in\mathbf{b}$ and its adjacent $b'\in \mathbf{b}$, where $b\ne b'$, or an endpoint, $u$, and its adjacent break point, $b_u$. This means that each tropical line segment, $\Gamma'_{b,b'}$, $\Gamma'_{u,b_u}$, and $\Gamma'_{v,b_v}$, comprising $\Gamma_{u,v}$ have no break points.

Consider the tropical line segment $\Gamma'_{u,b_u}$ and the vector $d$ where $d=u-b_u$. For a line segment with no break points, each coordinate, $d_i \in d$ is either zero or some value $\delta$. Specifically, $d_i \in \{0, \delta \}$. This means that any coordinates that change values when moving from $b_u$ to $u$, change by the same amount, $\delta$.  Otherwise a break point must exist due to the piecewise nature of a line segment in the tropical projective torus.

Now assume that any extension of $\Gamma'_{u,b_u}$ forming a new line segment $\Gamma'_{u^{'},b_u}$ containing the original endpoint $u$, must possess a break point. However if the following conditions exist, $\Gamma'_{u^{'},b_u}$ cannot contain a break point: 

\begin{enumerate}
    \item This new line segment contains $u$,
    \item $d_{i}^{'}\in \{0,\delta'\}$ where $d_{i}^{'} \in d^{'}$ and $d^{'}=u^{'}-b_u$,
    \item The non-zero values of $d^{'}$ are in the same coordinate positions as $d=u-b_u$.
\end{enumerate}
 Since no break point exists, $\Gamma_{u,b_u}$ may be extended without adding a break point.  Since $\Gamma_{u,b_u}$ can be extended without a break point, the line segment $\Gamma_{u,v}$ may also be extended without adding a break point.  By induction, as long as any extension follows the three conditions above, $\Gamma_{u,v}$ may be extended from its endpoints $u$ and $v$ infinitely without additional breakpoints, thus defining the tropical line containing $\Gamma_{u,v}$. 
\end{proof}

\begin{example}\label{ex:break}
Consider a line segment $\Gamma_{u,v} \in \mathbb R^3 \!/\mathbb R {\bf 1}$ where $u=(0,0,0)$ and $v=(0,3,1)$. Using Equation~\eqref{eq:troline}, we define $\Gamma_{u,v}$ with the following points: $(0,0,0); (0,2,0); (0,3,1)$ where $b_u = b_v = (0,2,0)$ as defined in Theorem~\ref{thm:ext}.  Note that $\Gamma_{u,v}$ is composed of two line segments, which is the maximum allowed for a line segment in $\mathbb R^3 \!/\mathbb R {\bf 1}$~\citep{joswigBook}.  This translates into $max(|\bf{b}|)=1$. We conclude then that any additional break points would define a new line segment such that $\Gamma_{u,v} \not \subset \Gamma_{u',v'}$.  But because $\Gamma_{u,v}$ is a tropical line segment it must be part of a tropical line.  Since $|\bf{b}|=1$, the tropical line must only contain the set of break points $\bf{b}$.  Therefore, any subset of the tropical line that includes $\Gamma_{u,v}$ can have no more than the original set of break points, $\bf{b}$, as well.  
\end{example}

Rules governing the resultant line segment extension in Example~\ref{ex:break} can be extended to any dimension.  Specifically, a line segment $\Gamma_{u,v} \in \mathbb R^e \!/\mathbb R {\bf 1}$ may only contain $e-1$ line segments, translating into $e-2$ break points, which does not include the end points~\citep{joswigBook}.

\begin{example}\label{ex:no_break}
Now consider a line segment $\Gamma_{u,v} \in \mathbb R^3 \!/\mathbb R {\bf 1}$ where $u=(0,0,0)$ and $v=(0,2,0)$. Using Equation~\eqref{eq:troline}, we see that $\Gamma_{u,v}$ possesses no break points which allows it to potentially be part of several different line segments since break points may be added. If we wish to extend $\Gamma_{u,v}$ in either or both directions without adding break points we need only consider the coordinate difference vector between $u$ and $v$.  If we wish to extend the line segment past $u$, we consider $d_u=u-v$ which in this case yields $d_u=(0,-2,0)$.  To extend the line segment without a break point we can do this by only changing the second coordinate since $d_u^2\not = 0$. For example $u'=b_u+\lambda*d_u=(0,-2\lambda,0)$ where $\lambda\in\mathbb{R}_{>1}$. If we then utilize Equation~\eqref{eq:troline} we find that $\Gamma_{u,v} \subset \Gamma_{u',v}$ where $\Gamma_{u',v}$ is an extension of $\Gamma_{u,v}$ without any additional break points.
\end{example}

\begin{figure}[H]
    \centering
    \includegraphics[width=0.44\textwidth]{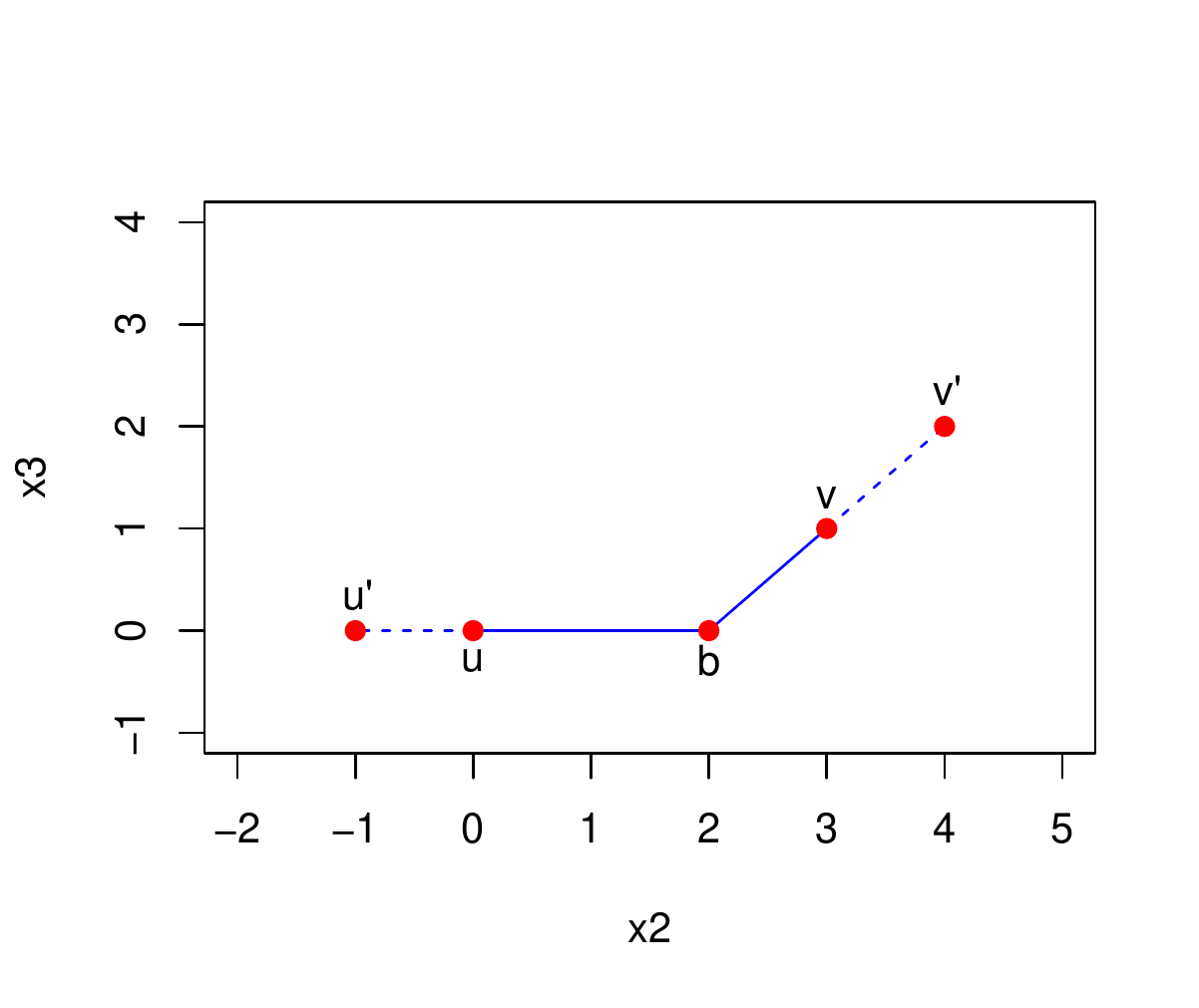}~
    \includegraphics[width=0.44\textwidth]{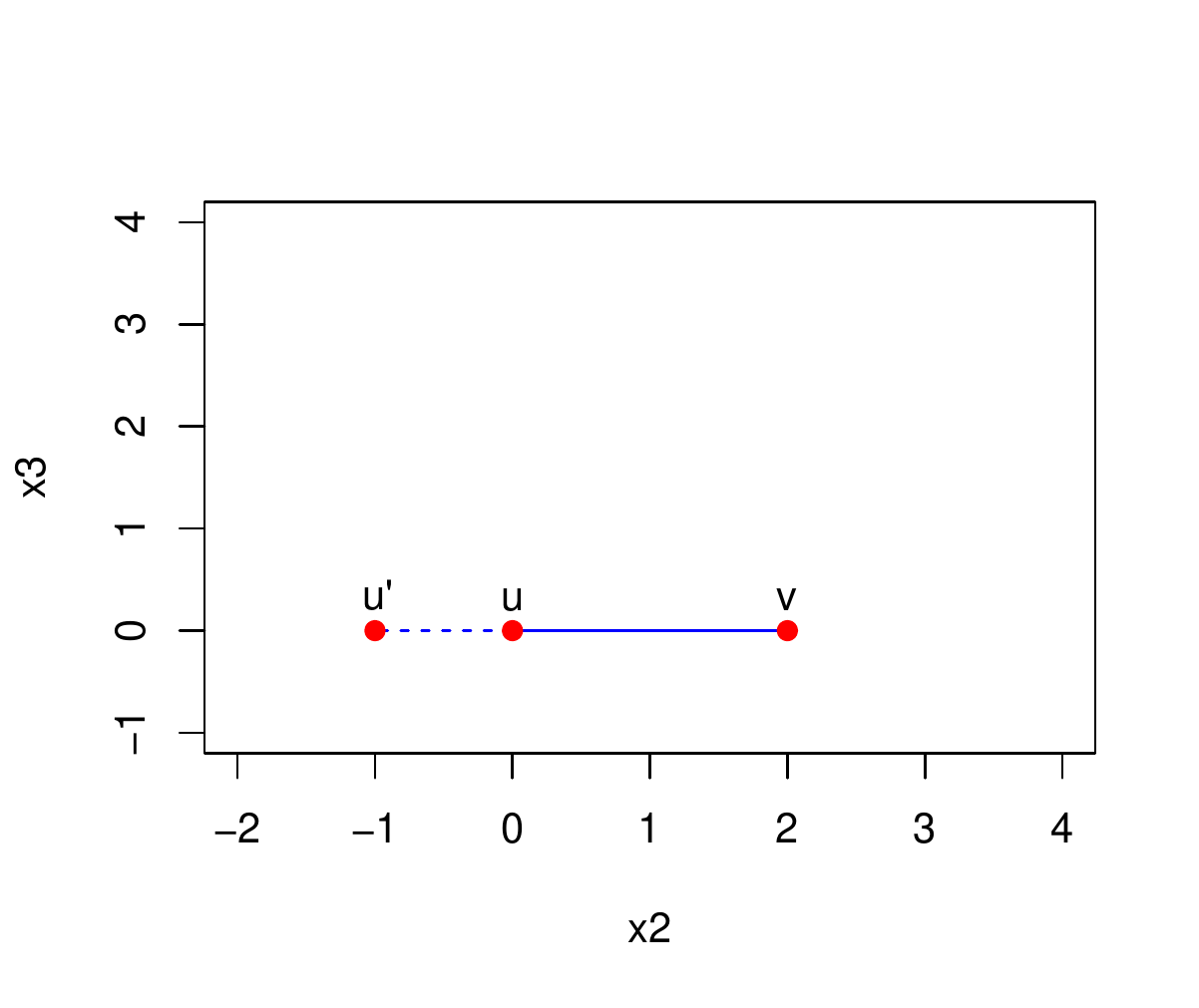}
    \caption{Possible extension results from Examples~\ref{ex:break} (left) and~\ref{ex:no_break} (right). In the both figures, points $u$ and $v$ are the end points for the associated line segment $\Gamma_{u,v}$ with the line segment defined by the solid blue line.  In the left figure, $b$ is the break point for the original line segment.  In both figures $u'$ and $v'$ (left figure only) represent the new end points after extending the line segments.}
    \label{fig:Ext_Ex}
\end{figure}

\subsection{Sampling from a Classical Polytope via Hit and Run (HAR) Algorithm}

In this section we illustrate the Hit-and Run (HAR) algorithm developed by~\cite{Zabinsky2013}. The HAR algorithm is a MCMC method to sample a point from a convex set, $S$ and consists of two main steps: 1) build a bidirectional line segment (or line) emanating from a point in a set $S$; 2) randomly select a point on that line ~\citep{Zabinsky2013}. Algorithm~\ref{alg:HAR1} defines the HAR MCMC.

This process repeats a number of times in order to ensure the starting point and ending point are decorrelated.  To ensure points are only selected from a particular set, $S$, it is often necessary to employ some variation of rejection sampling as we will see later.  The remainder of this paper will illustrate variations on Algorithm~\ref{alg:HAR1} for use over the tropical projective torus $\mathbb{R}^e/{\mathbb R} {\bf 1}$.

\begin{algorithm}[
H]
\caption{Sampling via HAR algorithm from $S$}\label{alg:HAR1}
\begin{algorithmic}
\State {\bf Input:} Initial point $x_0 \in S$ and maximum iteration $I \geq 1$.
\State {\bf Output:} A random point $x \in S$.
\State Set $k = 0$.
\For{$k= 0, \ldots , I-1$,}
\State Generate a random direction $D^k$ uniformly distributed over the surface of a unit hypersphere centered around $x_k$.
\State Generate a random point $x_{k+1}$ from a line $L_k:= \{y \in S: y = x_k + \lambda, \lambda \in D^k\}$.
\EndFor \\
\Return $x := x_{I}$.
\end{algorithmic}
\end{algorithm}

\subsection{Sampling from a Tropical Polytope}
In this section we introduce a HAR sampler for use on a tropical polytope, a tropical convex hull of finitely many points over the tropical projective torus $\mathbb{R}^e/{\mathbb R} {\bf 1}$.  Sampling points from a tropical polytope, $\mathcal{P}$, begins by defining an initial point $x_0\in \mathcal{P}$ then using variations of Algorithm~\ref{alg:HAR1} to find other points in $\mathcal{P}$.  Because a sampled point, $x$ may not fall inside of $\mathcal{P}$,  we accept or reject, $x$ by evaluating its tropical projection, $\pi(x)$ onto $\mathcal{P}$ as shown in Definition \ref{def:proj}. If $d_{tr}(x,\pi(x))=0$, then $x\in \mathcal{P}$ and we accept the proposed point. The algorithms proposed in this section leverage what we call {\em vertex HAR sampling} which we describe in the next section.  

\subsubsection{Vertex HAR Sampling with Tropical Line Extensions}\label{sec:vert_HAR_samp}

Recall that a HAR sampler applied to a classical convex hull in $\mathbb{R}^e$ iteratively samples from a line between an initial point and a point on the boundary of the convex hull.  This version of a HAR sampler from a tropical convex hull simply mimics the HAR sampler from a classical convex hull over Euclidean space, namely, it iteratively samples from a tropical line segment between an initial point and a point on the boundary of the tropical convex hull.   
Intuitively a tropical line segment between two vertices defines a two dimensional face of the tropical polytope $\mathcal{P}$ and $s-1 \geq \nu \geq 2$ vertices define a $\nu$ dimensional face of $\mathcal{P}$.  Therefore if we want to sample from a $\nu$ dimensional face of $\mathcal{P}$, we can apply the HAR sampler on $\nu$ many vertices of $\mathcal{P}$.    Using a sampled point from a $\nu$ dimensional face of $\mathcal{P}$ and the initial point, we can move to the next point using the HAR sampler.  Specifically, this algorithm leverages the vertex set $\{v^1,...,v^s \}$ of the tropical polytope, $\mathcal{P}$ by sampling points from a tropical line segment between the initial point and a point generated between $\nu$ randomly selected vertices.  Algorithm~\ref{alg:HAR_vert} illustrates this in detail.   

\begin{algorithm}[H]
\caption{Vertex HAR Sampling from $\mathcal{P}$ with $\nu  = 2$}\label{alg:HAR_vert} 
\begin{algorithmic}
\State {\bf Input:} Tropical polytope $\mathcal{P}:=\tconv(v^1, \ldots, v^s)$ and an initial point $x_0 \in \mathcal{P}$ and maximum iteration $I \geq 1$.
\State {\bf Output:} A random point $x \in \mathcal{P}$.
\State Set $k = 0$.
\For{$k= 0, \ldots , I-1$,}
\State Randomly select $v^{i_1}_k$ and $v^{i_2}_k$ such that $v^{i_1}_k,v^{i_2}_k \in \{v^1,...,v^s \}$ the vertex set of $\mathcal{P}$ where $i_1, i_2 \in \{1,\ldots, s\}$. 
\State Generate a random point $v$ from a tropical line segment $\Gamma^k_{v^{i_1}_k, v^{i_2}_k}$ using Algorithm \ref{eq:troline}.
\State Generate a random point $x_{k+1}$ from a tropical line segment $\Gamma^k_{x_0, v}$ using Algorithm \ref{eq:troline}.
\EndFor \\
\Return $x := x_{I}$.
\end{algorithmic}
\end{algorithm}


We use Algorithm~\ref{alg:HAR_vert} to run two experiments consisting of $1,000$ and $10,000$ samples, respectively, taken from the tropical polytope defined by the vertices $(0,0,0)$, $(0,3,1)$ and $(0,2,5)$, using a maximum iteration value $I=50$.  The results are shown in Figure~\ref{fig:MOD}.  

\begin{figure}[H]
    \centering
    \includegraphics[width=0.44\textwidth]{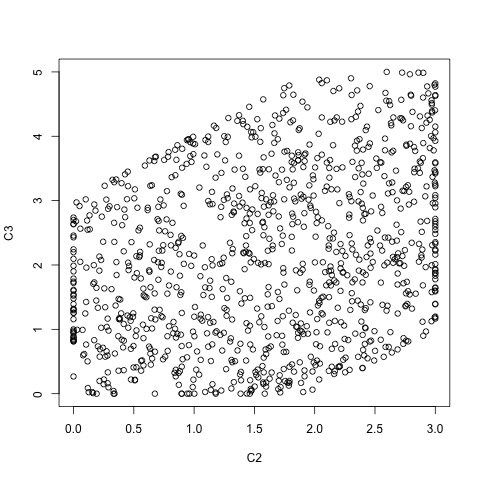}~
    \includegraphics[width=0.44\textwidth]{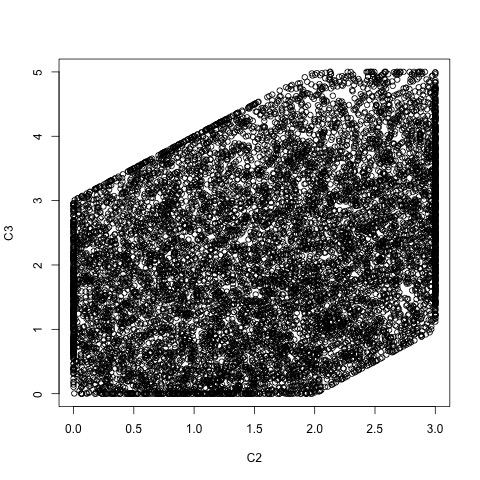}
    \caption{The result from an initial experiment for Algorithm \ref{alg:HAR_vert} with maximum iteration value $I=50$.  (Left) Results taken from $1,000$ samples. (Right) Results taken from $10,000$ samples.}
    \label{fig:MOD}
\end{figure}

 Algorithm~\ref{alg:HAR_vert} seems to fill the tropical space of the polytope $\mathcal{P}$ well though there appears to be some bias towards the edges of $\mathcal{P}$.  Repeated experiments suggest that sampling more often occurs in the lower portions $\mathcal{P}$.  Figure~\ref{fig:MOD} shows this characteristic in the fact that sampled points are sparser in the top right portion of $\mathcal{P}$ than in the rest of the polytope.

\begin{remark}
We can generalize the above $2$-vertex algorithm described in Algorithm~\ref{alg:HAR_vert} to $\nu$-vertex algorithm by repeating the uniform sampling from the line segments, which is described in Algorithm~\ref{alg:HAR_vert2}. That is, after we obtain the random point $x$ between $v^1$ and $v^2$, we next sample from the line segment connecting $x$ and $v^3$, where $v^3$ is another vertex.
\end{remark}

\begin{algorithm}[H]
\caption{Vertex HAR Sampling from $\mathcal{P}$ with $e-1 \geq \nu \geq 2$}\label{alg:HAR_vert2} 
\begin{algorithmic}
\State {\bf Input:} Tropical polytope $\mathcal{P}:=\tconv(v^1, \ldots, v^s)$ and an initial point $x_0 \in \mathcal{P}$ and maximum iteration $I \geq 1$.  The carinality value, $\nu$, of a subset of vertices to be chosen such that $s \geq \nu \geq 2$.
\State {\bf Output:} A random point $x \in \mathcal{P}$.
\State Set $k = 0$.
\For{$k= 0, \ldots , I-1$,}
\State Randomly select $v^{i_1}_k, \ldots , v^{i_{\nu}}_k$ such that $v^{i_1}_k, \ldots , v^{i_{\nu}}_k \in \{v^1,...,v^s \}$ the vertex set of $\mathcal{P}$.  
\State Generate a random point $v_0$ from a tropical line segment $\Gamma^k_{v^{i_1}_k, v^{i_2}_k}$, using Algorithm \ref{eq:troline}.
\For{$i =i_3, \ldots , i_{\nu}$,}
\State Generate a random point $v$ from a tropical line segment $\Gamma^k_{v_0, v^{i}_k}$, using Algorithm \ref{eq:troline}.
\State Set $v_0 = v$
\EndFor 
\State Generate a random point $x_{k+1}$ from a tropical line segment $\Gamma^k_{x_0, v_0}$ using Algorithm \ref{eq:troline}.
\EndFor \\
\Return $x := x_{I}$.
\end{algorithmic}
\end{algorithm}


We note that Figure~\ref{fig:MOD} indicates that while Algorithm~\ref{alg:HAR_vert} fills the space in the polytope, it biases on the edges of the polytope.  To combat this we make two modifications to Algorithm~\ref{alg:HAR_vert2}.  In its current form, sampling is limited to the line segment between a point, $v$, on the edge of the polytope, $\mathcal{P}$ and an initial point, $x_0$.  In this modification, we let $s= e$.  This will often define a line segment, $\Gamma_{x_0,v}$, where $v$ is likely on the edge of $\mathcal{P}$.

Because $\Gamma_{x_0,v}$ is limited in its length, we seek to extend it from its initial end points $x_0$ and $v$ without increasing the number of break points associated with $\Gamma_{x_0,v}$.  By extending the line $\Gamma_{x_0,v}$ to form a new line segment $\Gamma_{u',v'}$ where the end points $u'$ and $v'$ are at least on the edges of the polytope, $\mathcal{P}$, we increase the reachable points $\mathcal{P}$ (see Theorem~\ref{thm:ext}).  We also employ the extension in each iteration, $k \le I$ where $I$ is the number of iterations in the HAR algorithm. Algorithm~\ref{alg:HAR_ext} defines the extension of a tropical line segment without adding break points.

\begin{algorithm}
\caption{Tropical line segment, $\Gamma_{u,v}$, Extension}\label{alg:HAR_ext} 
\begin{algorithmic}
\State {\bf Input:} End points, $u$ and $v$ defining a tropical line segment, $\Gamma_{u,v}$; a scalar $d\in \mathbb{R}$ that defines the length of the segment between the new end points $u'$ and $v'$ and the break points, $b_u$ and $b_v$, immediately preceding each new end point.
\State {\bf Output:} Points $u'$ and $v'$ defining a new line segment $\Gamma_{u',v'}$ with the same break points as $\Gamma_{u,v}$.
\State Set $\delta_1=u-b_u$ and $\delta_2=v-b_v$
\State Define $u'=b_u+d*\delta_1$ and $v'=b_v+d*\delta_2$\\
\Return $u',v'$.
\end{algorithmic}
\end{algorithm}

Now we define Algorithm~\ref{alg:HAR_vert3} which 
says that the cardinality, $\nu=s$ where $s$ is the number of vertices of a given tropical polytope.  In addition, we utilize Algorithm~\ref{alg:HAR_ext} to extend line segments used to sample in $\mathcal{P}$.  This algorithm utilizes rejection sampling so it is more computationally expensive than previous vertex HAR algorithms but results are less biased to the edges of the polytope.

\begin{algorithm}
\caption{Vertex HAR Sampling from $\mathcal{P}$ with Extension and $\nu = s$}\label{alg:HAR_vert3} 
\begin{algorithmic}
\State {\bf Input:} Tropical polytope $\mathcal{P}:=\tconv(v^1, \ldots, v^s)$ and an initial point $x_0 \in \mathcal{P}$ and maximum iteration $I \geq 1$.  
\State {\bf Output:} A random point $x \in \mathcal{P}$.
\State Set $k = 0$.
\While{$k < I-1$,}

\State $k=k+1$
\State Randomly select $v^{i_1}_k$ and $v^{i_{2}}_k$ such that $v^{i_1}_k$ and $v^{i_{2}}_k \in \{v^1,...,v^s \}$ the randomly permuted vertex set of $\mathcal{P}$.  
\State Generate a random point $v_0$ from a tropical line segment $\Gamma^k_{v^{i_1}_k, v^{i_2}_k}$, using Algorithm \ref{eq:troline}.
\State $j=3$
\While{$j < \nu$,}

\State $j=j+1$
\State Extend the line segment $\Gamma^k_{v_0, v^{i_j}_k}$ to form a line segment $\Gamma_{u_j ',v_j '}$ with the same break points as $\Gamma^k_{v_0, v^{i_j}_k}$ using Algorithm~\ref{alg:HAR_ext}.
\State Generate a random point $v$ from a tropical line segment $\Gamma_{u_j ',v_j '}$, using Algorithm~\ref{eq:troline}.
\If{$v\in \mathcal{P}$ }

\State Set $v_0 = v$

\Else

\State $j=j-1$
\EndIf
\EndWhile
\State Extend the line segment $\Gamma^k_{x_k, v}$ to form a line segment $\Gamma_{u_k ',v_k '}$ with the same break points as $\Gamma^k_{x_k, v}$ using Algorithm~\ref{alg:HAR_ext}.
\State Generate a random point $x_{k+1}$ from a tropical line segment $\Gamma^k_{u_k, v_k}$ using Algorithm~\ref{eq:troline}.
 
\EndWhile\\
\Return $x := x_{I}$.
\end{algorithmic}
\end{algorithm}

\begin{remark}
Note that from our computational experiments, when $s > e$, then sampling via Algorithm \ref{alg:HAR_vert3} from a uniform distribution over a tropical polytope seems not well mixed.
\end{remark}

Because Algorithm~\ref{alg:HAR_vert3} is essentially defining a tropical line segment that spans a polytope, $\mathcal{P}$, which contains a starting point and then selects the successor, $x_1$, uniformly from that line segment, it seems intuitive that the $P(x_0,x_1)=P(x_1,x_0)$ since the line segment containing $x_0$ and $x_1$ is unique and is reversible.  This is not always the case as illustrated in the following example.

\begin{example}
Consider the tropical polytope, $\mathcal{P}$, defined by the vertices $V=\{(0,0,0),\;(0,3,1),\; (0,2,5)\}$ and consider two points, $x_0,x_1\in\mathcal{P}$, where $x_0=(0,2,2)$ and is the starting (input) point for Algorithm~\ref{alg:HAR_vert3} and $x_1=(0,1,3)$ which is the next point in the chain (output).  The goal is to show that $P(x_0,x_1)=P(x_1,x_0)$ and as we will see, this will only be proven if the probability density function associated with moving from $x_0$ to $x_1$ is the same as moving from $x_1$ to $x_0$. Figure~\ref{fig:algAB} shows $\Gamma_{A,B}$ containing $x_0$ and $x_1$.

\begin{figure}[H]
 \centering
 \includegraphics[width=0.45\textwidth]{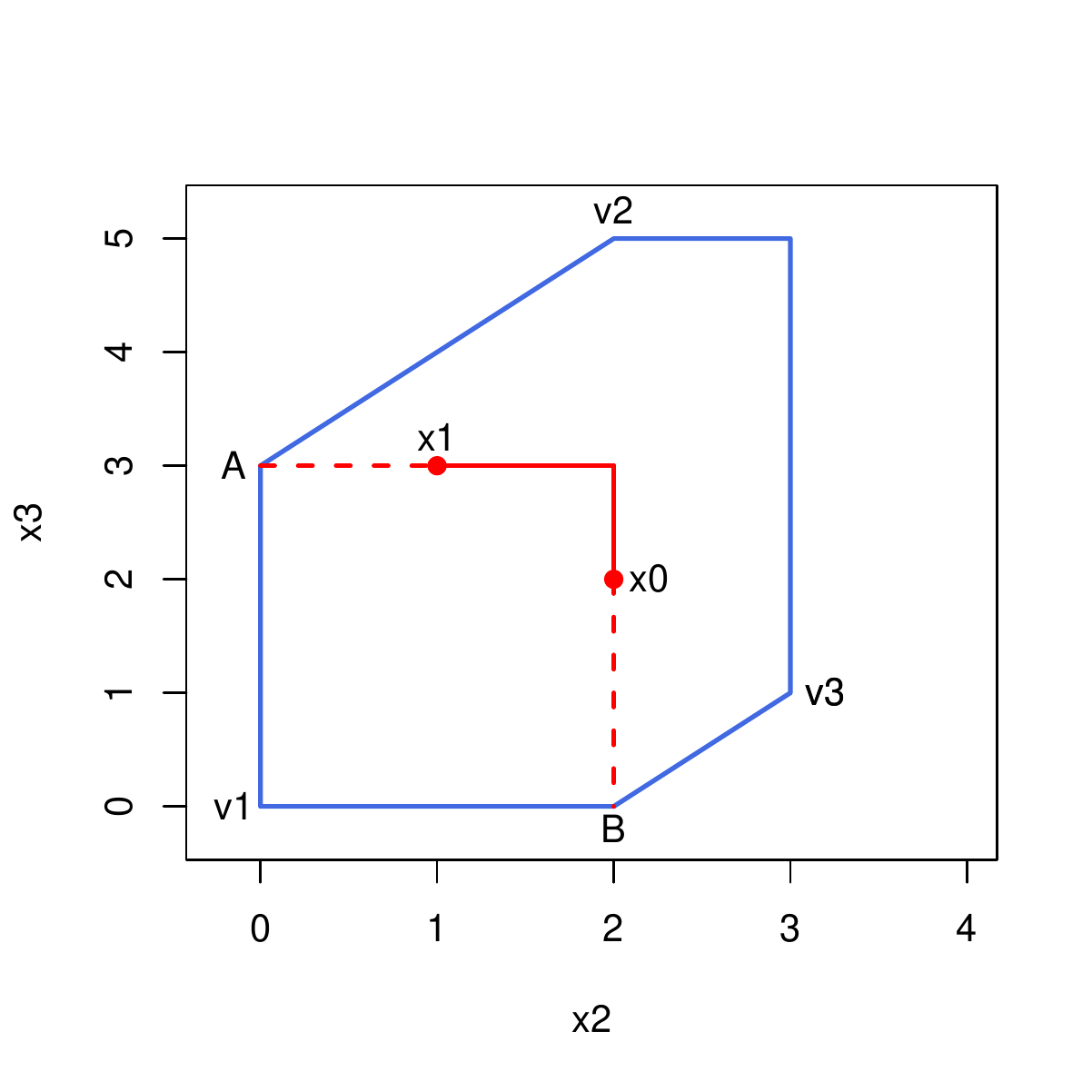} ~
\caption{The line segment $\Gamma_{A,B}$ containing $x_0$ and $x_1$.  Dashed lines indicate the extension to the boundary of $\mathcal{P}$ (blue lines).}
\label{fig:algAB}
\end{figure}

Specifically, we want to find the line segment, $\Gamma_{A,B}$ which spans the polytope, $\mathcal{P}$ such that $x_0,x_1 \in \Gamma_{A,B}$.  If we identify $\Gamma_{A,B}$, then it is possible to sample uniformly on $\Gamma_{A,B}$ and move from $x_0$ to $x_1$ as well as the converse.  Algorithm~\ref{alg:HAR_vert3} moves from $x_0$ to $x_1$ by successively drawing lines between points in $\mathcal{P}$ and vertices defining the tropical convex hull and extending the line segments to the boundary of $\mathcal{P}$. The initial line segment is drawn between two randomly chosen vertices, $u,v\in V$ and then a point is randomly chosen from this line. The probability of selecting any two vertices is $\frac{1}{\binom{n}{2}}$ where $n=|V|$, the cardinality of $V$ which, in this case, $n=3$, and the probability of choosing a pair of vertices from $V$ is $\frac{1}{3}$.  Each subsequent line segment is drawn from a point randomly chosen from the previous line segment to the next vertex and extending the line segment (if needed) to the boundary of $\mathcal{P}$. After the final line segment is constructed and a point sampled, a final line segment is drawn between the starting point and this final point.  The line segment is once again extended to the boundary of $\mathcal{P}$ to form $\Gamma_{A,B}$ and a point is sampled uniformly.

The orientation of $\Gamma_{A,B}$ to any line segment formed in Algorithm~\ref{alg:HAR_vert3} will determine the points on the line segments that lead to identifying $\Gamma_{A,B}$.  This results in either a single point sampled uniformly on a line segment or the choice of a point from an interval on a line segment that will result in reaching $\Gamma_{A,B}$. Figure~\ref{fig:alg6} shows the progression of line segments formed during an iteration that begins with $v^1$ and $v^2$ going from $x_0$ to $x_1$ and the reverse. 

\begin{figure}[H]
 \centering
 \includegraphics[width=0.43\textwidth]{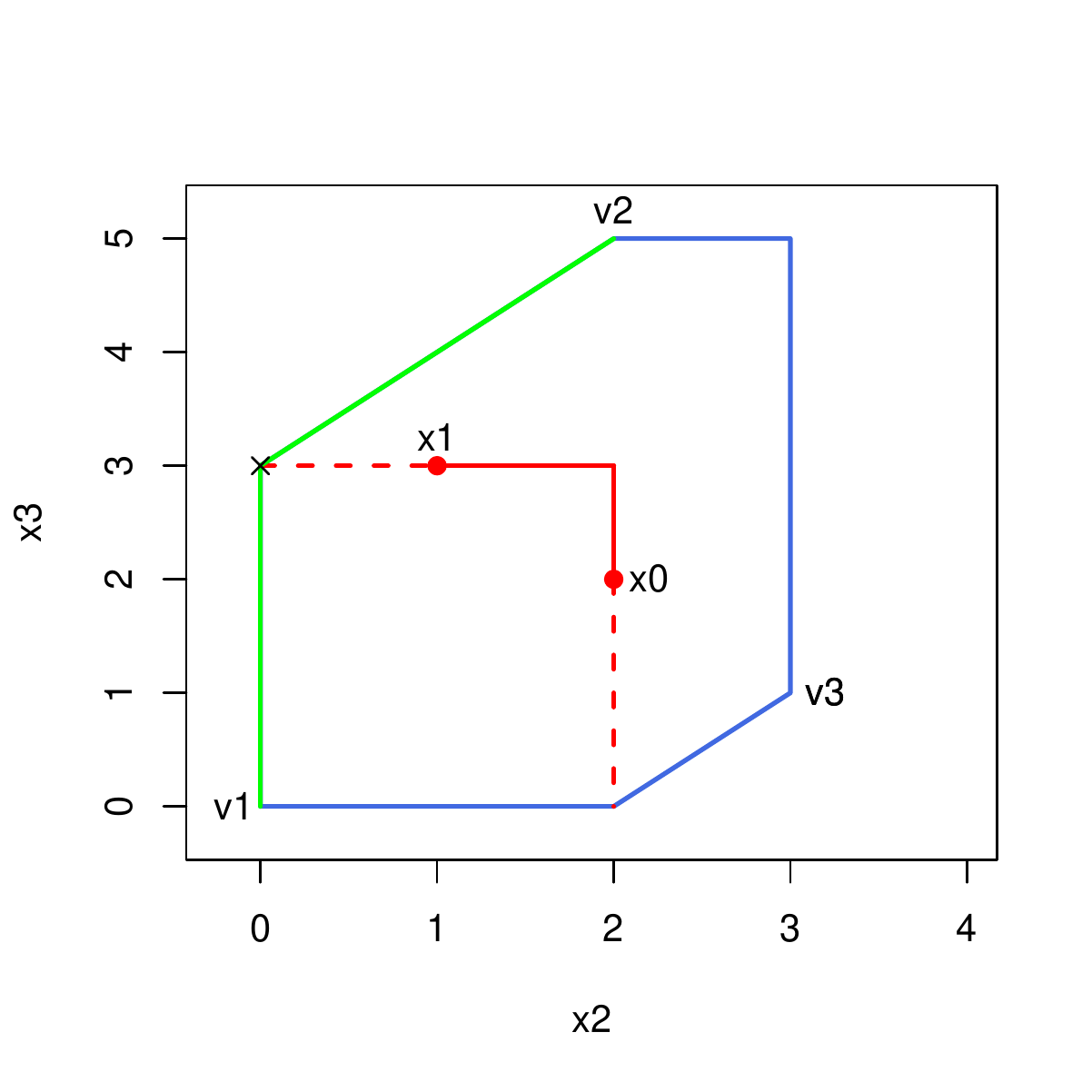} 
 \includegraphics[width=0.43\textwidth]{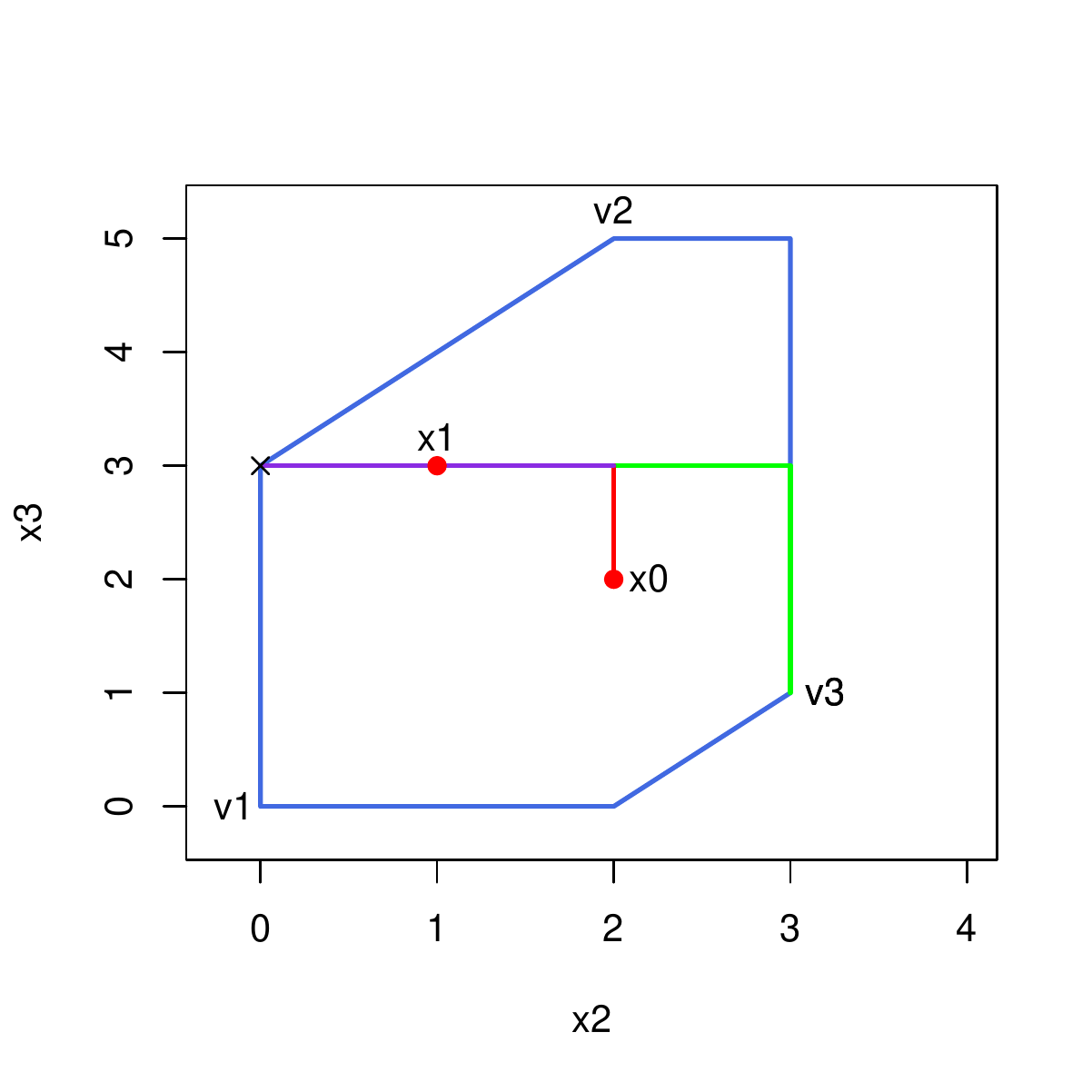} 
 \includegraphics[width=0.43\textwidth]{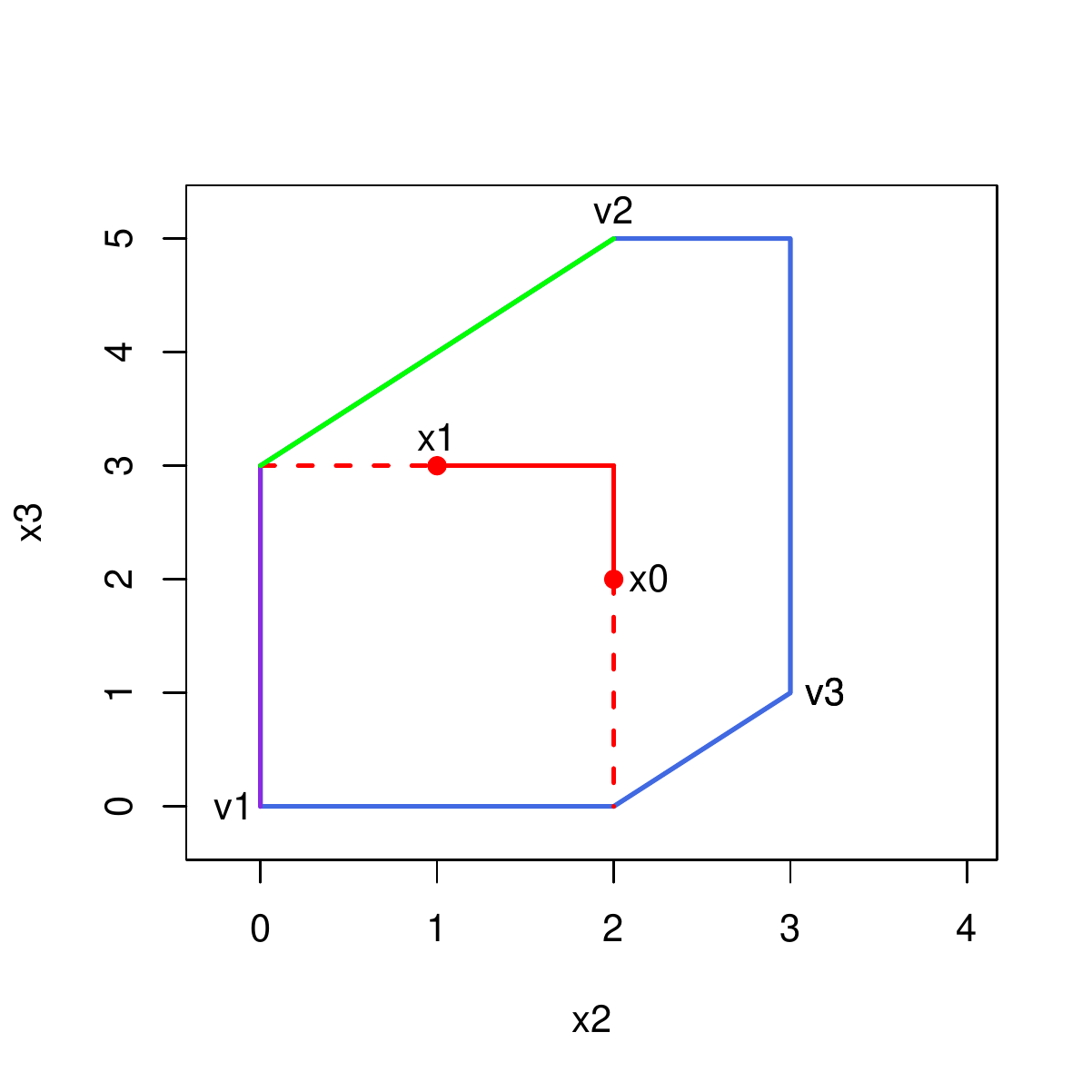} 
 \includegraphics[width=0.43\textwidth]{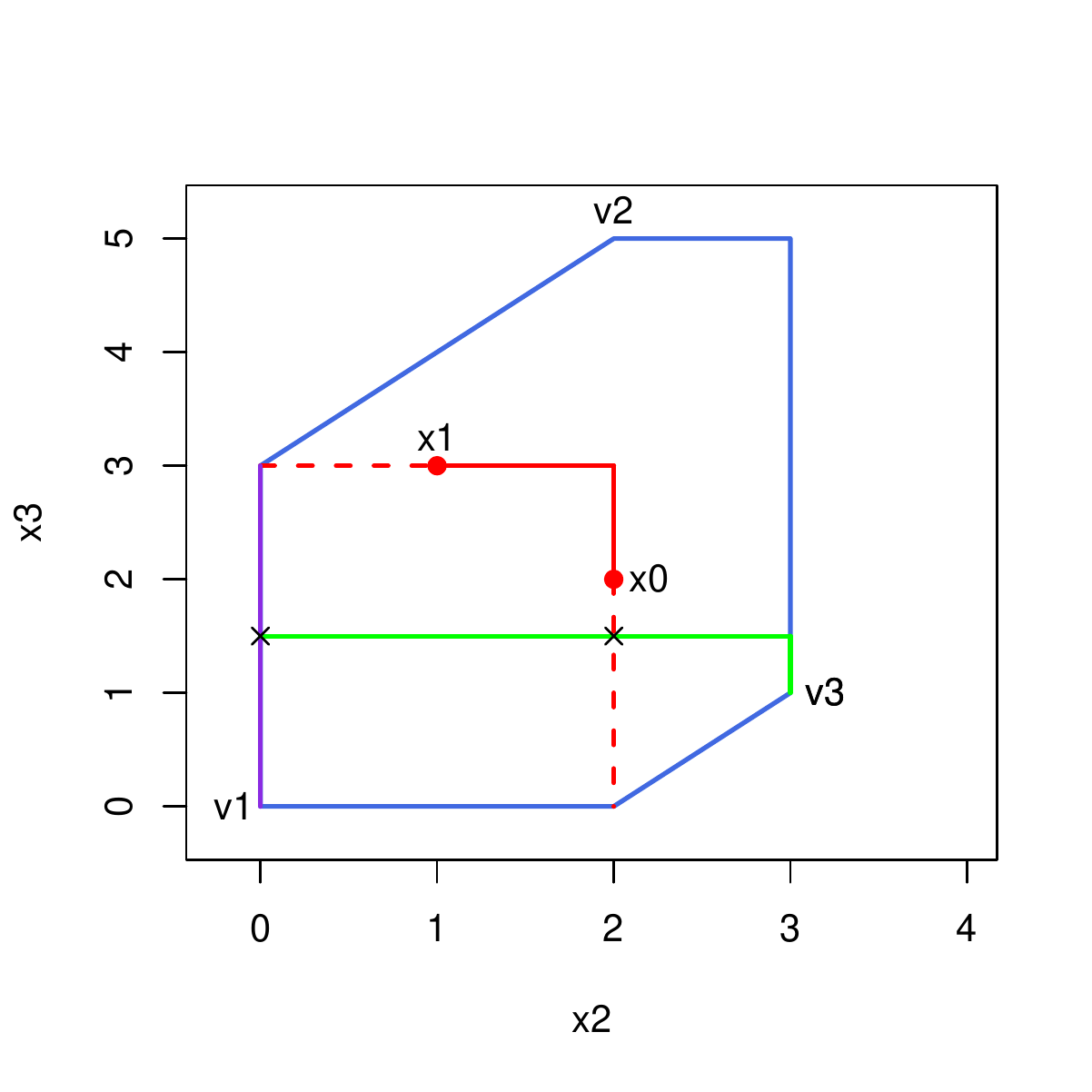} 
\caption{The progression of line segments in $\mathcal{P}$ moving from $x_0$ to $x_1$ (top row) and the reverse (bottom row), beginning with $v^1$ and $v^2$ that are necessary to build $\Gamma_{A,B}$ such that $x_0,x_1\in \Gamma_{A,B}$. The top left figure shows $\Gamma_{v^1,v^2}$ and $x\in\Gamma_{v^1,v^2}$ that will lead to $\Gamma_{A,B}$.  In the top right figure, we see an interval of points (purple portion) in $\Gamma_{x,v^3}$ that will result in identifying $\Gamma_{A,B}$. The bottom figures show the sequence to $\Gamma_{A,B}$ if $x_1$ is the initial point.  In either case, there are similar sequences with each pair of initial vertices.}
\label{fig:alg6}
\end{figure}

The pdfs $P(x_0,x_1)$ or $P(x_1,x_0)$ are a combination of sequences that identify $\Gamma_{A,B}$ starting from line segments defined by each pair of vertices. However, we end up sampling from line segments of differing tropical lengths moving from $x_0$ to $x_1$ versus $x_1$ to $x_0$ leading to $P(x_0,x_1)\ne P(x_1,x_0)$ meaning this is not symmetric.
\end{example}

\subsubsection{Vertex HAR Sampling with Extrapolation}
One reason that the Vertex HAR samplings are elegant and computationally convenient is that they essentially use only vertices and line segments without using external coordinates.
(Sampling a random point from the ambient space and then projecting it to a tropical convex polytope can lead to a biased sampling distribution as evident in Figure \ref{fig:tropPoly}.)
In order to realize the uniform sampling distribution, which can then be transformed to any desired distribution, it may be beneficial to further proceed this intrinsic geometry nature.
To make it easy to compute the transition probabilities, it may help to extend a tropical line segment in a more symmetric way, which we call ``extrapolation'' to distinguish from the ``extension'' in the previous algorithms.
Specifically, in a tropical polytope $\mathcal{P}:=\tconv(v^1, \ldots, v^s)$ in $\mathbb{R}^e/\mathbb{R}{\bf 1}$, an ``extrapolated'' point from $v^i$ through a point $x$ is defined by the projection to the other vertices or $\mathcal{P}^{-i}:=\tconv(v^1, \ldots, v^{i-1}, v^{i+1}, \ldots, v^s)$ as
\begin{equation}\label{eq:tropproj2}
\pi_{\mathcal{P}^{-i}} (x) := 
\bigoplus_{\substack{l=1 \\ l\neq i}}^s \lambda_l \odot v^l~ ~ {\rm where} ~ ~ \lambda_l \!=\! {\rm min}(x-v^l). 
\end{equation}

\begin{lemma}
$\pi_{\mathcal{P}^{-i}} (x) \in \mathcal{P}$.
\end{lemma}
\begin{proof}
A point in $\mathcal{P}$ is represented as $\sum_{l=1}^{s} a_l \odot v^l$ where $a_l \in \mathbb{R}$. $\pi_{\mathcal{P}^{-i}} (x)$ is obtained as a special point close to the limit $a_i \to -\infty$.
\end{proof}
\begin{lemma}
A point $x \in \mathcal{P}$ is on the line segment between $w_i$ and $\pi_{\mathcal{P}^{-i}} (x)$.
\end{lemma}
\begin{proof}
The line segment connecting $w_i$ and $\pi_{\mathcal{P}^{-i}} (x)$ is represented as $a_i \odot v^i \oplus a_{-i} \odot \pi_{\mathcal{P}^{-i}} (x)$ where $a_{i}, a_{-i} \in \mathbb{R}$.
$x$ is obtained when $a_{i} = \lambda_i$ and $a_{-i}=0$.
\end{proof}
That is, we can regard $\pi_{\mathcal{P}^{-i}} (x)$ as an ``extrapolated'' point.
Algorithm \ref{alg:HAR_extrapolation} samples uniformly from the entire line segment connecting $\pi_{\mathcal{P}^{-i}} (x)$ and $v^i$ where $x$ is a given initial point.

\begin{algorithm}
\caption{Vertex HAR Sampling from $\mathcal{P}$ using extrapolation with $\nu = 1$}\label{alg:HAR_extrapolation}
\begin{algorithmic}
\State {\bf Input:} Tropical polytope $\mathcal{P}:=\tconv(v^1, \ldots, v^s)$ and an initial point $x_0 \in \mathcal{P}$ and maximum iteration $I \geq 1$.
\State {\bf Output:} A random point $x \in \mathcal{P}$.
\State Set $k = 0$.
\For{$k= 0, \ldots , I-1$,}
\State Randomly select $v^i_k$ such that $v^i_k \in \{v^1,...,v^s \}$ the vertex set of $\mathcal{P}$.
\State Generate a random point $x_{k+1}$ from a tropical line segment $\Gamma^k_{w^i_k, \pi_{\mathcal{P}^{-i}} (x_k)}$ using Algorithm \ref{eq:troline}.
\EndFor \\
\Return $x := x_{I}$.
\end{algorithmic}
\end{algorithm}

\begin{example}[Extrapolation]
Consider the tropical polytope generated by three vertices $(0, 0, 0), \, (0, 3, 1), \, (0, 2, 5)$ in $\mathbb R^3 \!/\mathbb R {\bf 1}$  (Figure \ref{fig:tropPoly} (top right)).
We apply the extrapolation method to the point $x = (0, 2, 2) \in \mathbb{R}^3/\mathbb{R}{\bf 1}$ for which $(\lambda_1, \lambda_2, \lambda_3)=(0,-1,-3)$.
That is, $x = \lambda \cdot v$.
Then
$\pi_{\mathcal{P}^{-1}} (x)=(0, 3, 3)$,
$\pi_{\mathcal{P}^{-2}} (x)=(0, 0, 2)$,
$\pi_{\mathcal{P}^{-3}} (x)=(0, 2, 0)$.
\end{example}
\begin{example}[Extrapolation with Additional Vertex]
If we add $w_4=(0,4,6)$ to the previous example, for $\mathcal{P'}:=\tconv(v^1, v^2, v^3, v^4)$, $\lambda_4=-4$ and $x = \lambda \cdot v$.
Then
$\pi_{\mathcal{P'}^{-1}} (x)=(0, 3, 3)$,
$\pi_{\mathcal{P'}^{-2}} (x)=(0, 0, 2)$,
$\pi_{\mathcal{P'}^{-3}} (x)=(0, 2, 2)$,
$\pi_{\mathcal{P'}^{-4}} (x)=(0, 2, 2)$.
\end{example}

The following proposition shows the connectivity of Markov chain via Algorithm \ref{alg:HAR_extrapolation} over the tropical polytope $\mathcal{P} = \tconv(v^1, \ldots , v^s) \subset \mathbb{R}^e / \mathbb{R} {\bf 1}$.
\begin{proposition}\label{pro:connectivity}
Any two points $u, v \in \mathcal{P} = \tconv(v^1, \ldots , v^s) \subset \mathbb{R}^e / \mathbb{R} {\bf 1}$ are connected via Algorithm \ref{alg:HAR_extrapolation}.
\end{proposition}
\begin{proof}
We prove this proposition by induction on $s$.  Suppose $s = 2$.  Then $\mathcal{P} = \tconv(v^1, v^2)$, which is a tropical line segment. Therefore it is trivial.
Suppose that any two points $x, y \in \tconv(v^1, \ldots , v^s) \subset \mathbb{R}^e / \mathbb{R} {\bf 1}$ are connected via Algorithm \ref{alg:HAR_extrapolation} for some $s > 2$.
Then for any two points $u, v \in \tconv(v^1, \ldots , v^s, v^{s+1})$,
there is a route via Algorithm \ref{alg:HAR_extrapolation} as $u \rightarrow \pi_{\mathcal{P}^{-(s+1)}} (u) (\in \tconv(v^1, \ldots , v^s)) \rightarrow \pi_{\mathcal{P}^{-(s+1)}} (v) (\in \tconv(v^1, \ldots , v^s)) \rightarrow v$.
\end{proof}

Given that any two points in $\mathcal{P}$ are connected after finite steps of Algorithm \ref{alg:HAR_extrapolation} by Proposition \ref{pro:connectivity}, 
the following proposition on the symmetric proposal distribution indicates that Algorithm \ref{alg:HAR_extrapolation} usually leads to the uniform distribution.
\begin{proposition} \label{prop:symmetric_proposal_distribution}
Let $\mathcal{P}$ be a tropical polytope with a set of vertices $\{v^1, \ldots , v^s\}$ $\subset \mathbb{R}^e / \mathbb{R} {\bf 1}$.
Let $x_1 \in \mathcal{P}$ be sampled by Algorithm \ref{alg:HAR_extrapolation} with the initial point $x_0 \in \mathcal{P}$ and $I = 1$.
Then, if there is a unique line segment connecting them that can be sampled by Algorithm \ref{alg:HAR_extrapolation} both with $x_0$ and $x_1$ as an initial point, the proposal distribution $P(x_0, x_1)$ is symmetric.
That is, 
the proposal distribution from $x_0$ to $x_1$, $P(x_0,x_1)$, and that from $x_1$ to $x_0$, $P(x_1,x_0)$ are equal.
\end{proposition}
\begin{proof}
Let $\Gamma_{A,B}$ be the unique tropical line segment, that passes through A and B, which can be sampled by Algorithm \ref{alg:HAR_extrapolation}. Then,
\begin{equation}
P(A,B) 
= p(\Gamma_{A,B}|A) \frac{1}{ d_{\rm tr}(v^i,\pi_{\mathcal{P}^{-i}} (A))}
\end{equation}
is equal to
\begin{equation}
P(B,A) 
= p(\Gamma_{A,B}|B) \frac{1}{ d_{\rm tr}(v^i,\pi_{\mathcal{P}^{-i}} (B))}
\end{equation}
where $p(\Gamma_{A,B}|x) = \frac{1}{s}$ is the probability of choosing the line segment that passes through A and B in HAR at $x$ and
$\pi_{\mathcal{P}^{-i}} (A) = \pi_{\mathcal{P}^{-i}}(B) $.
\end{proof}
\begin{remark}
If there is no chance that a line segment that passes through A and B is sampled by Algorithm \ref{alg:HAR_extrapolation}, $P(A,B) = P(B,A) = 0$. In any cases, $P(A,B) = P(B,A)$.
\end{remark}

\begin{example}\label{ex:counter_ex}
Even if you can go from A to B or $P(A,B) > 0$, it is possible that you cannot go from B to A or $P(B,A) = 0$.
This is because $x_k$ and $x_{k+1}$ chosen in Algorithm \ref{alg:HAR_extrapolation} are not necessarily projected to the same point.
You can find a counter example in Figure \ref{fig:lambda3}. There $\pi_{\mathcal{P}^{-2}}(0,-\epsilon,0)=(0,0,-1), \pi_{\mathcal{P}^{-1}}(0,-\epsilon,0) = \pi_{\mathcal{P}^{-3}}(0,-\epsilon,0) = (0,-\epsilon,0)$ and $\pi_{\mathcal{P}^{-2}}(0,\epsilon,0)=(0,\epsilon,\epsilon-1)$ for $0 \leq \epsilon \leq 1$.
This means that you can go, for example, from $(0,\epsilon,0)$ to $(0,-\epsilon,0)$ but you cannot go in the other way.
Thus, the extrapolation is not unique in the sense that different line segments can pass through and share two given points.
\end{example}

\begin{figure}
 \centering
 \includegraphics[width=0.4\textwidth]{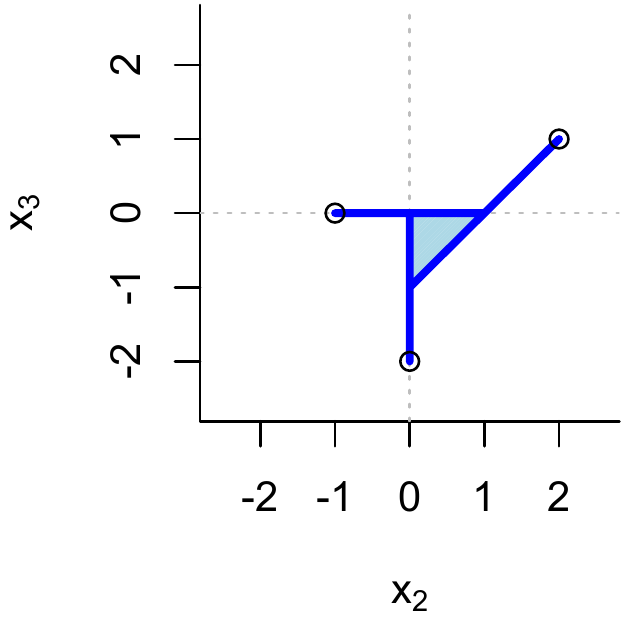}
\caption{The tropical polytope of three points $(0, 2, 1), \, (0, -1, 0), \, (0, 0, -2)$ in $\mathbb R^3 \!/\mathbb R {\bf 1}$ is colored light blue. The three vertices are denoted by the black circles and the line segments connecting the vertices are emphasized by the thick blue lines. Example~\ref{ex:counter_ex} shows when using Algorithm~\ref{alg:HAR_extrapolation}, we can construct line segments $\Gamma_{A,B}$ but we cannot construct $\Gamma_{B,A}$ with certainty since $x_k=(0,\epsilon,0)$ and $x_{k+1}=(0,-\epsilon,0)$, for $0 \leq \epsilon \leq 1$, do not project to the same point.}
\label{fig:lambda3}
\end{figure}

Combining Proposition \ref{pro:connectivity} and Proposition \ref{prop:symmetric_proposal_distribution} guarantees the uniform sampling distribution from a polytrope that is defined as a tropical polytope that is also convex in the ordinary sense.
\begin{theorem}
Algorithm \ref{alg:HAR_extrapolation} samples points uniformly from a (full-dimensional) tropical polytrope $\mathcal{P} = \tconv(v^1, \ldots , v^e) \subset \mathbb{R}^e / \mathbb{R} {\bf 1}$.
\end{theorem}
\begin{proof}
According to Theorem 15 in~\citep{DS}, every polytrope (obtained as a result of the decomposition of a tropical polytope) is specified by a type $S$ as an intersection of sectors of tropical hyperplanes whose apices are the vertices of the original tropical polytope, as stated in~\citep{polytrope}.
Thus, for $x \in \mathcal{P}$, $\lambda_l = {\rm min}(x-v^l) = x_{S(l)} - v^l_{S(l)}$, where the minimum is attained only at $S(l)$.
As $S(i) \neq S(j)$ for $i \neq j$ and $S(1)\cup \ldots \cup S(e) = \{1, \ldots , e\}$, $x$ and $\lambda$ (or $\pi_{\mathcal{P}}$) are in one-to-one correspondence.
Suppose $A, B \in \mathcal{P}$.
Then, $\pi_{\mathcal{P}^{-i}}(A) \neq \pi_{\mathcal{P}^{-i}}(B)$ if and only if $v_i$, $A$ and $B(\neq A)$ are not on the same tropical line segment.
Thus when you can go from $A$ to $B$ via the line segment $\Gamma_{v_i,\pi_{\mathcal{P}^{-i}}(A)}$ in Algorithm \ref{alg:HAR_extrapolation}, you can also go from $B$ to $A$ via the same line segment.
Furthermore it is the unique line segment that connects $A$ and $B$ directly.
\end{proof}

Extrapolation does not necessarily require a vertex $v^i$ as an endpoint.
We can instead choose a subset $U \subset \{v^1,...,v^s \}$ and denote the projection to $\tconv(U)$ by $\pi_{\tconv (U)}$
and the projection to $\tconv(\{v^1,...,v^s \} \setminus U)$ by $\pi_{\tconv(\{v^1,...,v^s \} \setminus U)}$.
\textbf{\begin{lemma}
A point $x \in \mathcal{P}$ is on the line segment between $\pi_{\tconv (U)}(x)$ and $\pi_{\tconv(\{v^1,...,v^s \} \setminus U)}(x)$.
\end{lemma}}
\begin{proof}
The line segment connecting $\pi_{\tconv (U)}(x)$ and $\pi_{\tconv(\{v^1,...,v^s \} \setminus U)}(x)$ is represented as $a_U \odot \pi_{\tconv (U)}(x) \oplus a_{-U} \odot \pi_{\tconv(\{v^1,...,v^s \} \setminus U)}(x)$ where $a_U, a_{-U} \in \mathbb{R}$.
$x$ is obtained when $a_U = a_{-I} = 0$.
\end{proof}
Again, we can regard $\pi_{\tconv (U)}(x)$ and $\pi_{\tconv(\{v^1,...,v^s \} \setminus U)}(x)$ as ``extrapolated'' points.
Algorithm \ref{alg:HAR_extrapolation2} samples uniformly from the ``extrapolated'' line segment connecting them with $x$ as an initial point.

\begin{algorithm}[H]
\caption{Vertex HAR Sampling from $\mathcal{P}$ utilizing extrapolation with unfixed $\nu$} \label{alg:HAR_extrapolation2}
\begin{algorithmic}
\State {\bf Input:} Tropical polytope $\mathcal{P}:=\tconv(v^1, \ldots, v^s)$ and an initial point $x_0 \in \mathcal{P}$ and maximum iteration $I \geq 1$.
\State {\bf Output:} A random point $x \in \mathcal{P}$.
\State Set $k = 0$.
\For{$k= 0, \ldots , I-1$,}
\State Randomly select a non-empty set $U^k$ such that $U^k \subset \{v^1,...,v^s \}$ the vertex set of $\mathcal{P}$.
\State Generate a random point $x_{k+1}$ from a tropical line segment $\Gamma^k_{\pi_{\tconv (U)} (x_k), \pi_{\tconv(\{v^1,...,v^e \} \setminus U)} (x_k)}$ using Algorithm \ref{eq:troline}.
\EndFor \\
\Return $x := x_{I}$.
\end{algorithmic}
\end{algorithm}

Since Algorithm \ref{alg:HAR_extrapolation2} is a generalization of Algorithm \ref{alg:HAR_extrapolation}, we have the following lemma.
\begin{proposition}\label{pro:connectivity2}
Any two points $u, v \in \mathcal{P} = \tconv(v^1, \ldots , v^s) \subset \mathbb{R}^e / \mathbb{R} {\bf 1}$ are connected via Algorithm \ref{alg:HAR_extrapolation2}.
\end{proposition}
\begin{proof}
Algorithm \ref{alg:HAR_extrapolation2} has more routes than Algorithm \ref{alg:HAR_extrapolation}.
Therefore, if $u$ and $v$ are connected via Algorithm \ref{alg:HAR_extrapolation}, they are also connected via Algorithm \ref{alg:HAR_extrapolation2}.
\end{proof}

\begin{proposition}\label{prop:symmetric_proposal_distribution2}
Let $\mathcal{P}$ be a tropical polytope with a set of vertices $\{v^1, \ldots , v^s\}$ $\subset \mathbb{R}^e / \mathbb{R} {\bf 1}$.
Let $x_1 \in \mathcal{P}$ be sampled by Algorithm \ref{alg:HAR_extrapolation} with the initial point $x_0 \in \mathcal{P}$ and $I = 1$.
Then, if there is a unique line segment connecting them that can be sampled by Algorithm \ref{alg:HAR_extrapolation2} both with $x_0$ and $x_1$ as an initial point, the proposal distribution $P(x_0, x_1)$ is symmetric.
That is, the proposal distribution from $x_0$ to $x_1$, $P(x_0,x_1)$, and that from $x_1$ to $x_0$, $P(x_1,x_0)$ are equal.
\end{proposition}
\begin{proof}
The same as the proof for Proposition \ref{prop:symmetric_proposal_distribution}, but with $p(\Gamma_{A,B}|A) = p(\Gamma_{A,B}|B) = \frac{2}{2^s-2}$.
\end{proof}

Using Proposition \ref{pro:connectivity2} and \ref{prop:symmetric_proposal_distribution2}, we have the following theorem.
\begin{theorem}
Algorithm \ref{alg:HAR_extrapolation2} samples points uniformly from a (full-dimensional) tropical polytrope $\mathcal{P} = \tconv(v^1, \ldots , v^e) \subset \mathbb{R}^e / \mathbb{R} {\bf 1}$.
\end{theorem}

\begin{remark}
For the initial point $x_0 \in \mathcal{P}$, there are only finite ways ($s$-way in Algorithm \ref{alg:HAR_extrapolation} and $\frac{2^s-2}{2}$-way in Algorithm \ref{alg:HAR_extrapolation2}) to choose the line segments in HAR, because there are finite ways to divides $s$ vertices into two groups and every line segment is sampled equally.
Note that while you can choose any directions in classical HAR, in Algorithms \ref{alg:HAR_extrapolation} and \ref{alg:HAR_extrapolation2} you are allowed to choose only $s$ or $\frac{2^s-2}{2}$ directions.
This is never problematic as any random walk engine basically works for MCMC.
Rather, having finite choices makes it easy to impose equal probabilities to all the choices.
\end{remark}

\begin{remark}\label{rem:dist}
The stationary probability density function $p(x)$ should satisfy the detailed balance, $p(x=A) P(A,B) = p(x=B) P(B,A)$, for $A, B \in \mathcal{P}$. Thus $p(x) = c$ for some $c \geq 0$, which defines the uniform distribution.

When the stationary probability density function $p(x)$ that is not the uniform distribution is desired, you can use the Metropolis-Hasting algorithm with the acceptance probability $r = \min \left\{\frac{p(x_{k+1})}{p(x_k)}, 1\right\}$ for the point $x_{k+1}$ that is sampled according to the proposal distribution $P(x_{k},x_{k+1})$.
\end{remark}
We utilize Algorithms~\ref{alg:HAR_extrapolation} on a polytope $\mathcal{P}$ described in Example 3.15.  Sampling results are shown in Figure~\ref{fig:Ex_Ex}.

\begin{figure}[H]
    \centering
    \includegraphics[width=0.44\textwidth]{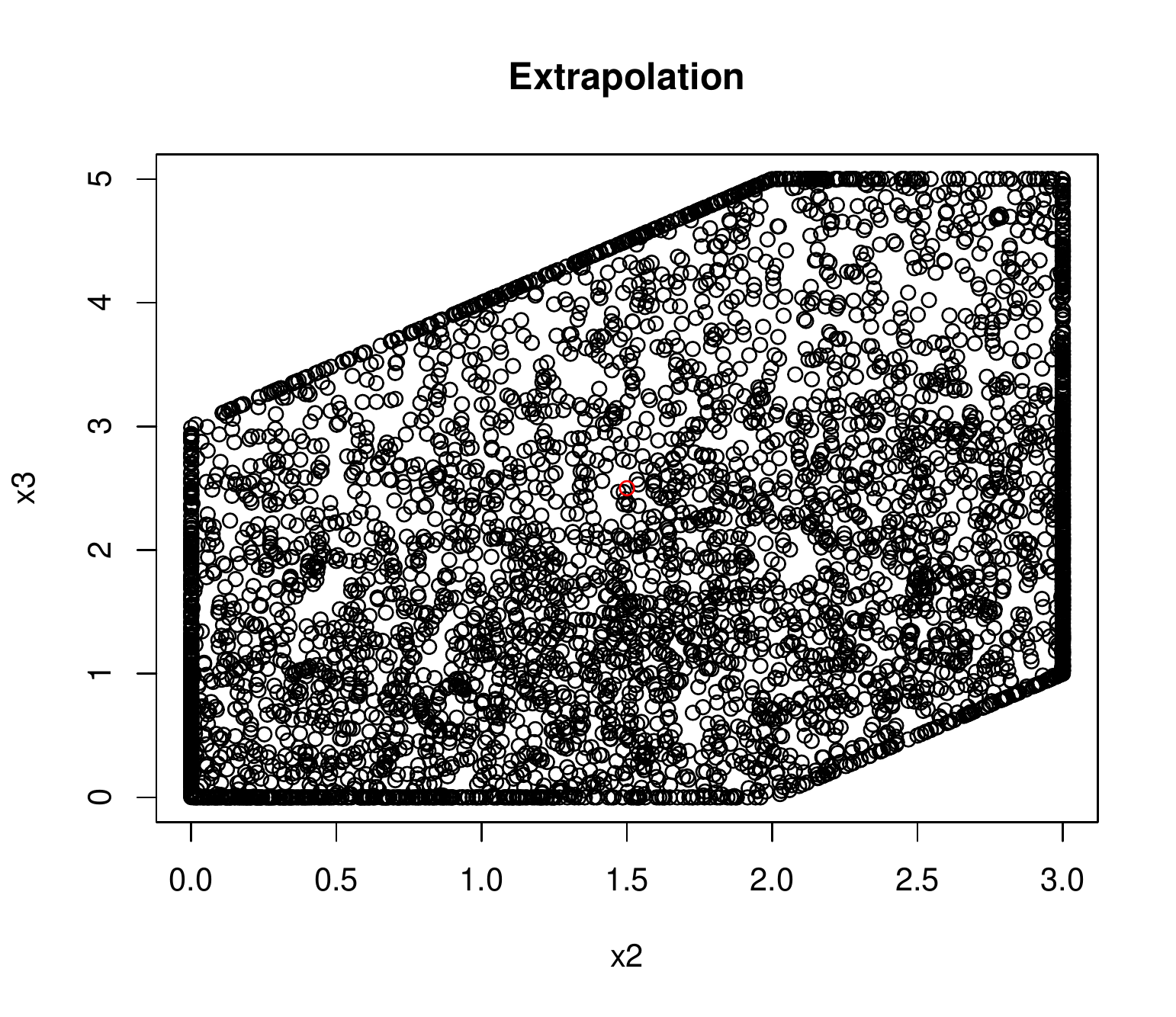}~
    \includegraphics[width=0.44\textwidth]{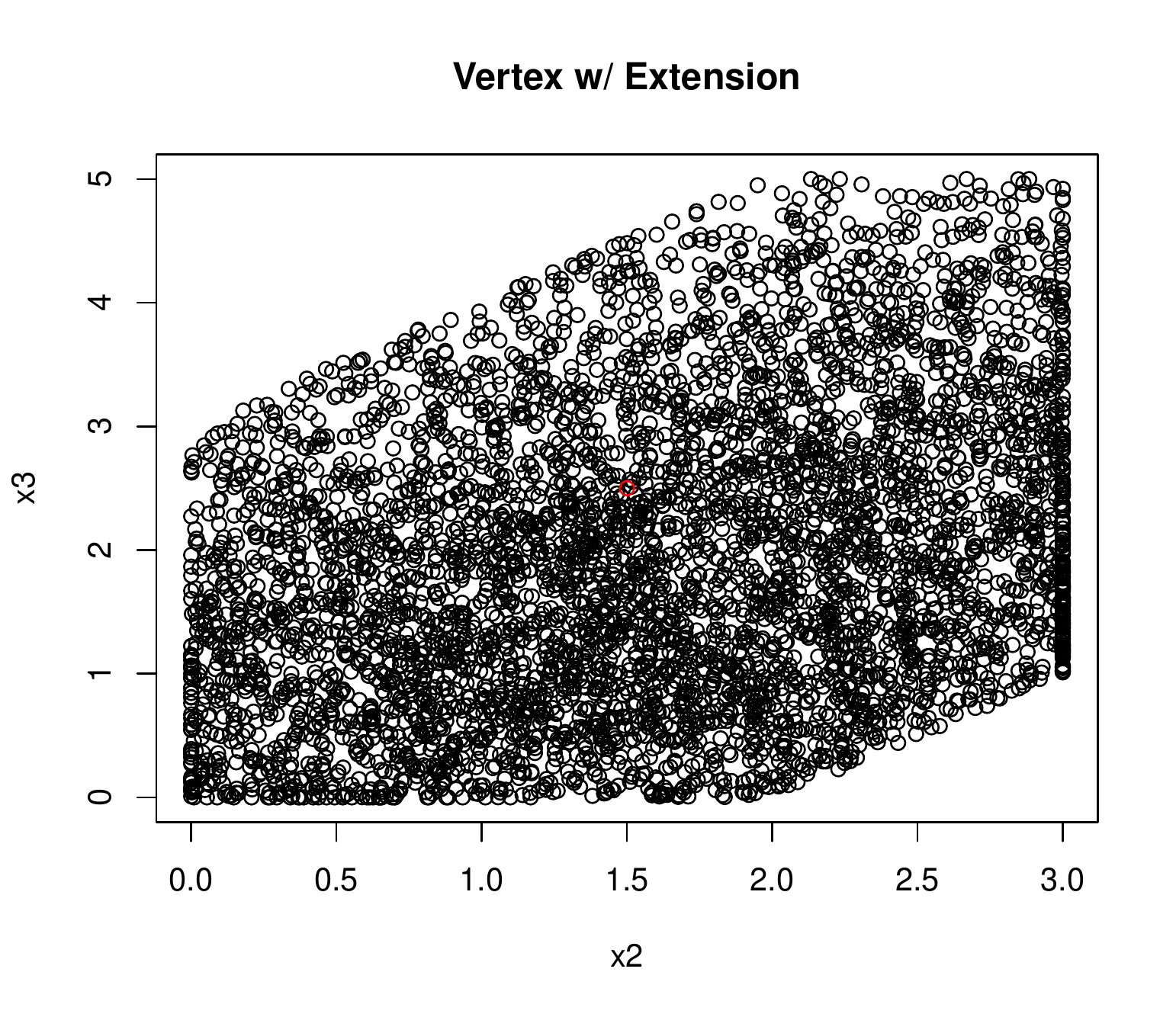}
    \caption{The result from comparative experiments using Algorithms \ref{alg:HAR_extrapolation2} (left) \ref{alg:HAR_vert3} (right) sampling 5,000 points each from a tropical polytope $\mathcal{P}$.  For these experiments we used a ``burn-in'' value of $1,000$ and $I=500$ for each algorithm.  For Algorithm~\ref{alg:HAR_vert3} we allowed a line extension scalar, $d=10$, ensuring any line segment extended beyond the edges of $\mathcal{P}$. The {\color{red}red} point is the starting point in each case.}
    \label{fig:Ex_Ex}
\end{figure}


Observe that for the extrapolation results, sampled points often occur on the boundaries of $\mathcal{P}$ and, similarly to Algorithm~\ref{alg:HAR_vert2}, sampling in the upper portion of the polytope is sparser than the lower portion.  
Using the extension algorithm (Algorithm~\ref{alg:HAR_vert3}), results suggest sampling is less concentrated on the boundaries of $\mathcal{P}$ (though still present in certain areas) compared with extrapolation results (Figure~\ref{fig:Ex_Ex} (right)).  Further, while sampling occurs more often in the lower portion of the polytope than the upper, it the areas are much less sparse than in the extrapolation results.

\subsubsection{Sampling from a Tropical Polytope Using a Given Distribution}

Sampling according to a specific distribution is useful in many circumstances for classical polytopes.  In this section, we propose an HAR algorithm which samples from a tropical polytope that mimics other MCMC algorithms that sample from classical polytopes according to a normal distribution, $N(0,\sigma)$. Remark~\ref{rem:dist} shows how this can be done in the general case.

The normal distribution utilizes the idea of a Euclidean ball, $B_r(x)$, around a point $x\in \mathbb{R}^e$ with radius, $r>0$ which is defined as,

\[
B_r(x) = \{y \in \mathbb{R}^e\; |\; ||x-y||_2 \leq r\}.
\]

\noindent The radius of this ball is defined by the standard deviation, $\sigma$, where the majority of points sampled according to $N(0,\sigma)$ fall within this ball.

We can use the idea of a ball in tropical space to develop an analogous sampling method to that used when sampling from a normal distribution in Euclidean space.  A tropical ball, $B_r(x)_{tr}$, around $x \in \mathbb{R}^e/\mathbb{R}{\bf 1}$ with a radius $r > 0$ is defined as follows:
\[
B_r(x)_{tr} = \{y \in \mathbb{R}^e/\mathbb{R}{\bf 1} | d_{\rm tr}(x, y) \leq r\}
\]

\noindent where $d_{\rm tr}(x, y)$ is the tropical distance between a point $y$ and a center of mass point $x$.
\begin{example}
Suppose we have $\mathbb{R}^3/\mathbb{R}{\bf 1}$ and let $\mu = (0, 1/2, 1/2) \in \mathbb{R}^e/\mathbb{R}{\bf 1}$ and let $r = 1/4$.  Then 
\[
B_{1/4}((0, 1/2, 1/2))_{tr} = \{y \in \mathbb{R}^3/\mathbb{R}{\bf 1} | d_{\rm tr}((0, 1/2, 1/2), y) \leq r\}
\]
and it is shown in Figure \ref{fig:tropball}.
\begin{figure}[H]
    \centering
    \includegraphics[width=0.8\textwidth]{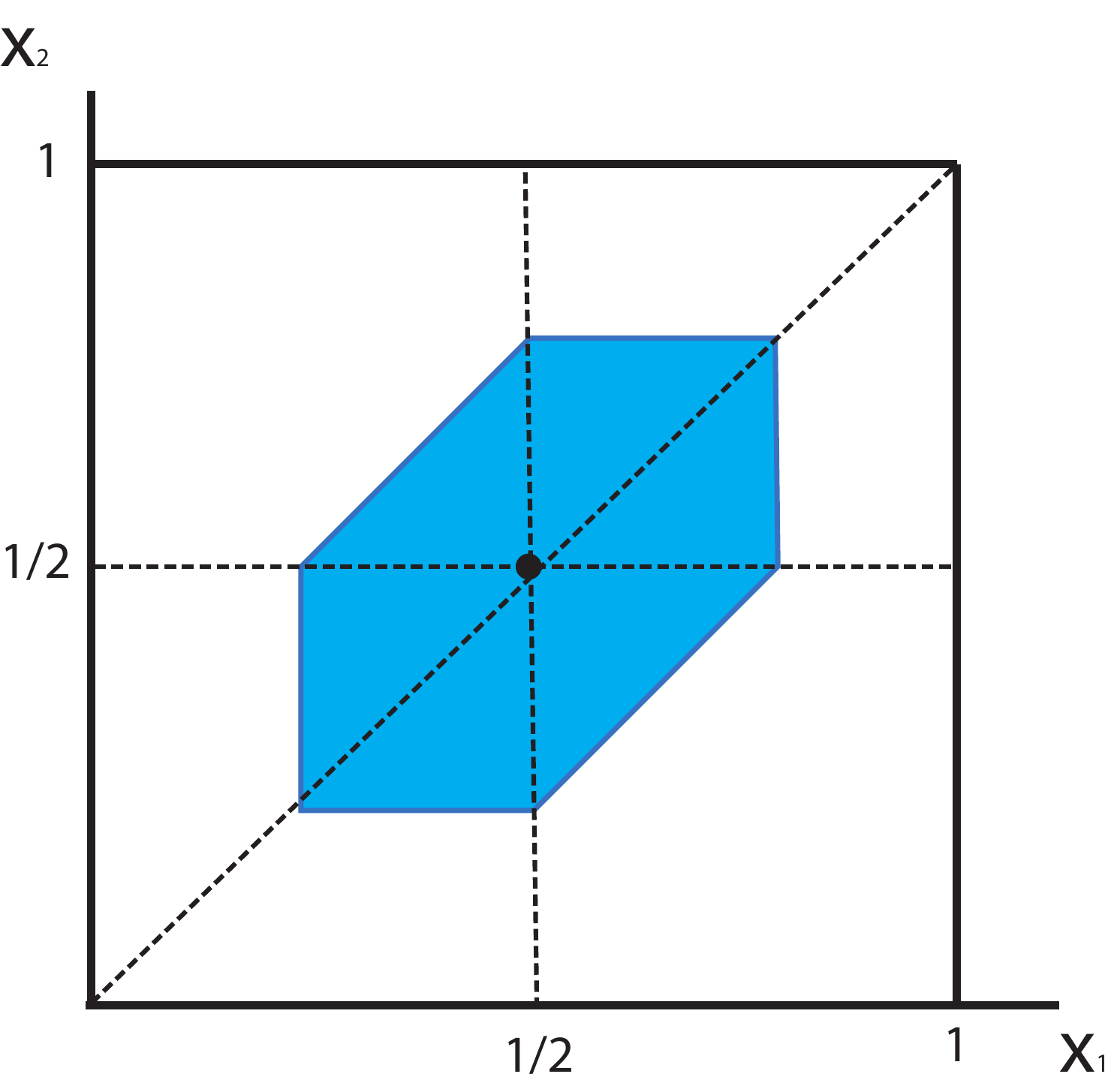}
    \caption{Tropical ball $B_{1/4}((0, 1/2, 1/2))_{tr}$.}
    \label{fig:tropball}
\end{figure}
\end{example}

Now we introduce a method to sample from a tropical polytope that mimics Gaussian ($N(0,\sigma)$) sampling from a classical polytope.  First, we define a classical Gaussian sampler in Algorithm~\ref{alg:Gauss1}.

\begin{algorithm}[H]
\caption{Gaussian HAR sampling from a polytope $\mathcal{S}$}\label{alg:Gauss1}
\begin{algorithmic}
\State {\bf Input:} Initial point $x_0 \in \mathcal{S}$; center of mass point $\mu \;\in \mathcal{S}$; number of iterations, $I$; a variance, $\sigma^2$; a large scalar value $t\ge \max\limits_{x,y}||x-y||_2$ where $x,y\in \mathcal{S}$.
\State {\bf Output:} A random point $x \sim \mathcal{N}(0,\sigma)$ sampled from $S$.
\State Set $k = 0$.
\While{$k< I-1$}{
\State $k=k+1$
\State Generate a random direction $D_k$ uniformly distributed over the surface of a unit hypersphere centered around $x_k$. 
\State Define $\pi(\mu)_k$ as the projection of the point $\mu$ onto the line defined by $D_k \cdot t$.

\State Generate $\lambda \sim \mathcal{N}(0,\sigma)$.

\State Generate a random point $x'$ from a line $L_k:= \{y \in S: y = \pi(\mu)_k + \lambda D'_k,\; \lambda \sim \mathcal{N}(0,\sigma)\}$.

\EndWhile} \\
\Return $x:=x_I$.
\end{algorithmic}
\end{algorithm}

Using Algorithm~\ref{alg:Gauss1}, we show an example of drawing $2,000$ sampled points from the vicinity of a central point $\mu=(1.5, 2.5)$ and $\sigma^2=3$ with no defined polytope.  The results are shown in Figure~\ref{fig:HAR_NORM}.

\begin{figure}[H]
    \centering
    \includegraphics[width=0.5\textwidth]{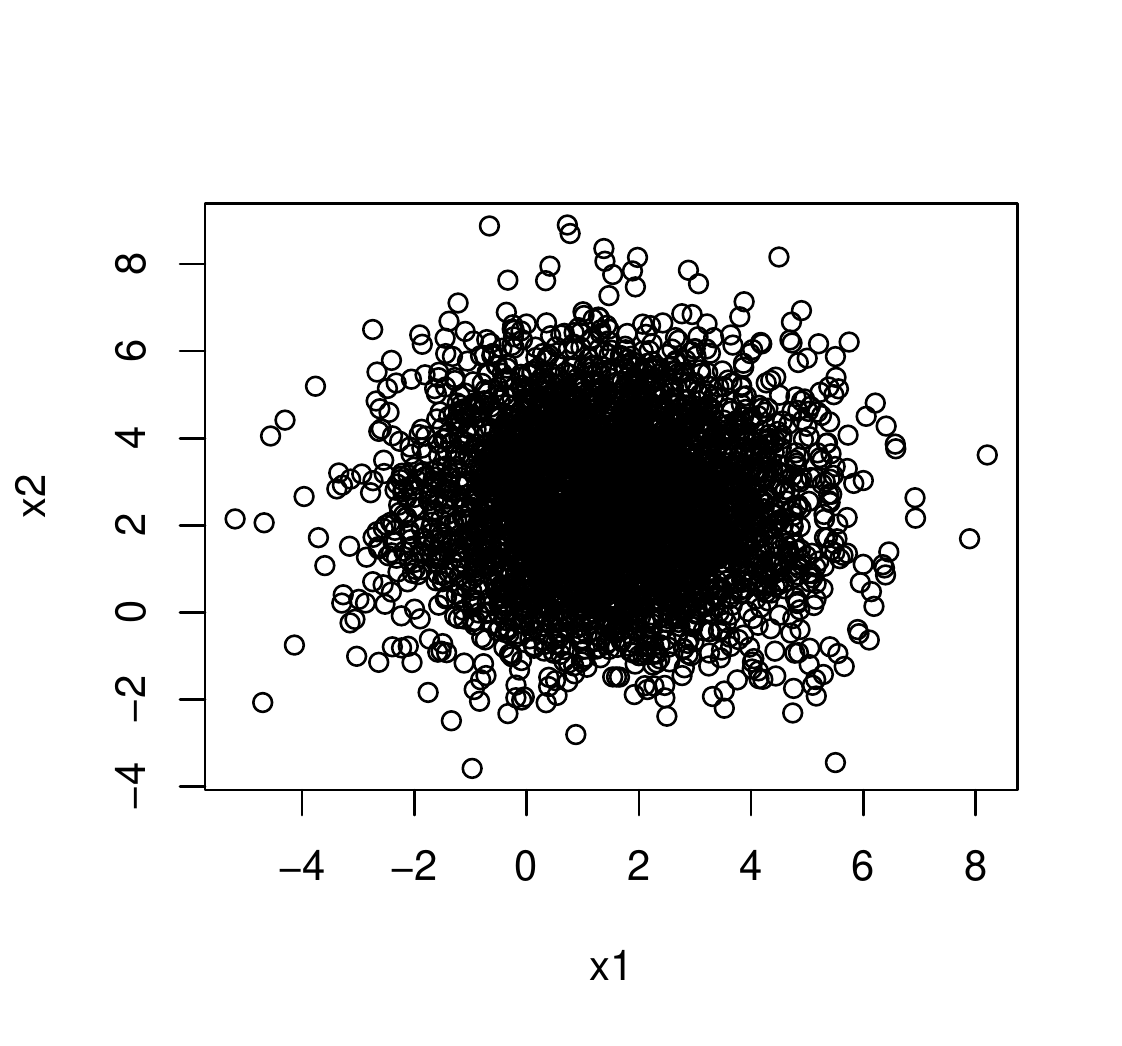}~
    \includegraphics[width=0.5\textwidth]{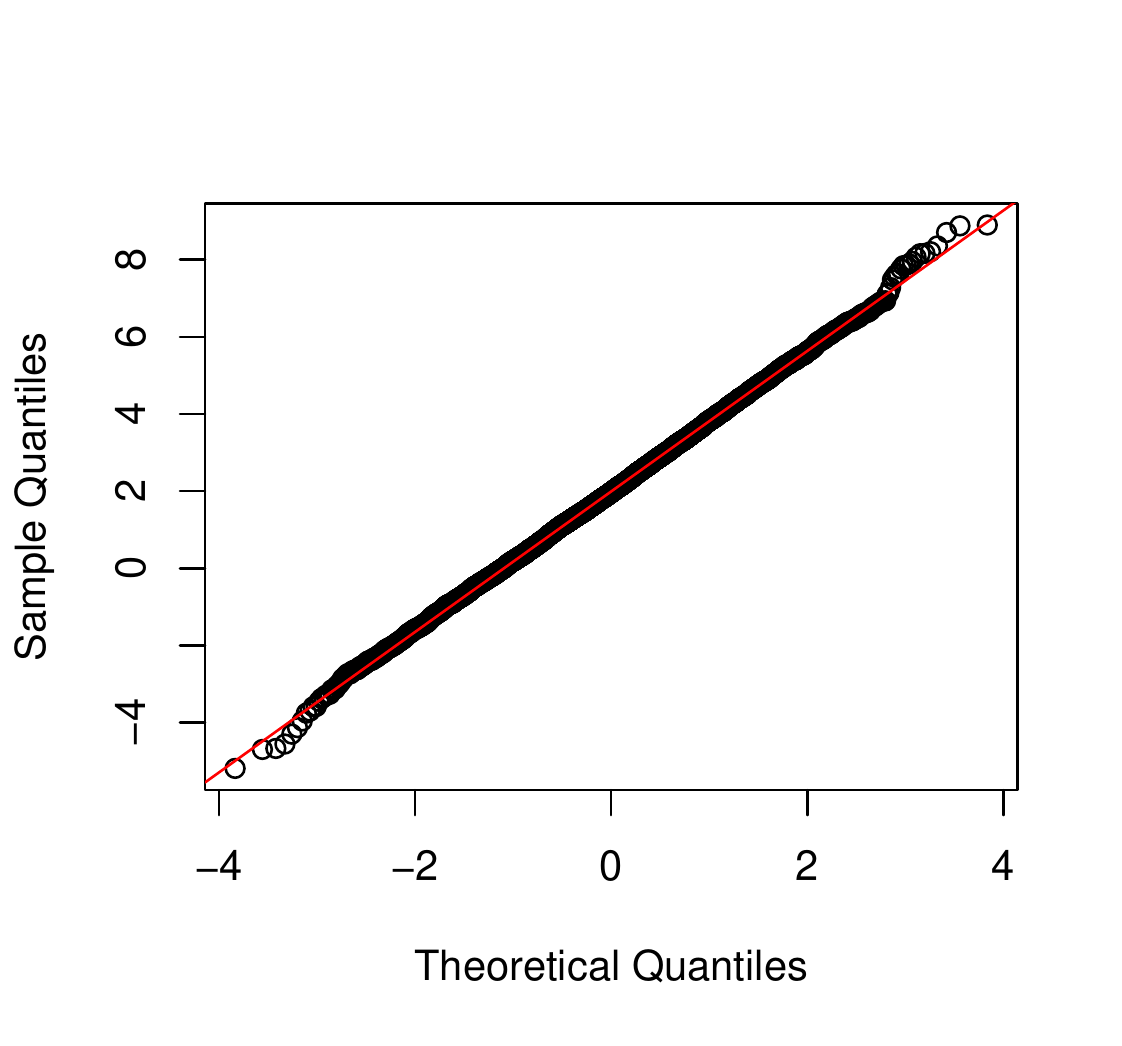}
    \caption{The result from an initial experiment for Algorithm \ref{alg:Gauss1} with maximum iteration value $I=50$ and variance, $\sigma^2=3$.  (Left) Results taken from $2,000$ samples. (Right) Resultant Quantile-Quantile plot to establish normality.}
    \label{fig:HAR_NORM}
\end{figure}

As expected we see the heavy concentration around the center of mass point, $\mu$.  The Quantile-Quantile plot demonstrates the Gaussian nature of the sampled points.

In a similar fashion we can sample from a tropical polytope where sampled points are concentrated around a center of mass point, $\mu$.  In both the classical and tropical cases, we use a parameter to define the a ball, $B_r(x)$ or $B_r(x)_{tr}$, respectively, which ultimately controls the dispersion of sampled points.  In the classical case, we are simply defining a standard deviation, $\sigma$.  In the tropical case, we define a tropical distance, which we define as $\sigma_{tr}$ which in turn defines a tropical ball, $B_r(x)_{tr}$ centered on $\mu$. In Figure~\ref{fig:tropball}, we see the structure of any $B_r(x)_{tr}$ in $\mathbb{R}^3 /\mathbb{R}\textbf{1}$.

It may be tempting to assume that in tropical space that sampled points concentrated around a center of mass point, $\mu$, are Gaussian in distribution.  However, this is not true given the structure of $B_r(x)_{tr}$.  For this and reasons related to tropical projections, we must leverage methods that differ from those used in Algorithm~\ref{alg:Gauss1}.

 To apply a HAR algorithm to a tropical polytope and concentrate sampling around a point $\mu$, we use the tropical metric over a tropical convex set $\mathcal{P}$ in $\mathbb{R}^e / \mathbb{R}{\bf 1}$. In addition, we combine this with Metropolis-Hastings filtering to sample from a distribution, $f(x)$, that is proportional to either of the following 
\begin{equation}\label{eq:gaus1}
    f(x)\propto \exp\left(\frac{-d_{\rm tr}(\mu, x)}{\sigma_{tr}}\right) dx
\end{equation}

\begin{equation}\label{eq:gaus2}
    f(x) \propto \exp\left(\frac{-d_{\rm tr}(\mu, x)^2}{\sigma_{tr}}\right) dx.
\end{equation}

 Algorithm~\ref{alg:COM} demonstrates how the tropical metric and Metropolis-Hastings filtering are employed to sample in the vicinity of a point $\mu$. While either~(\ref{eq:gaus1}) or~(\ref{eq:gaus2}) will work, we use~(\ref{eq:gaus2}) for reasons that will follow shortly.

\begin{algorithm}
\caption{Center of mass HAR sampling from a tropical polytope $S$}\label{alg:COM}
\begin{algorithmic}
\State {\bf Input:} Center of mass point $\mu\;\in \mathcal{P}$; $\sigma_{tr} \in \mathbb{R}$ defining a tropical ball $\mathcal{B}_{tr}$; sample size $n$.
\State {\bf Output:} A sample $\{x_1, \ldots , x_n\}$ sampled from $\mathcal{P}$ with distribution proportional to \eqref{eq:gaus1}.
\State Set $x_0 = \mu$.
\While{$i \leq n$}
\State Sample a proposal $x^*$ from $\mathcal{P}$ via a HAR sampler with the tropical metric. 
\State Set $r = \min \left\{\frac{\exp\left(\frac{-d_{\rm tr}(\mu, x^*)^2}{\sigma_{tr}}\right)}{\exp\left(\frac{-d_{\rm tr}(\mu, x_0)^2}{\sigma_{tr}}\right)}, 1\right\}$.
\State With probability $r$, accept the proposal, i.e., set $x_0 = x^*$, $x_i = x_0$, and set $i = i + 1$.
\EndWhile \\
\Return $\{x_1, \ldots , x_n\}$.
\end{algorithmic}
\end{algorithm}

Figure~\ref{fig:THAR_NORM} shows the results of an experiment with 3,000 sampled points using (\ref{eq:gaus1}) (left) and (\ref{eq:gaus2}) (right) in the Metropolis-Hasting filtering portion of Algorithm~\ref{alg:COM}. 

\begin{figure}[H]
    \centering
    \includegraphics[width=0.5\textwidth]{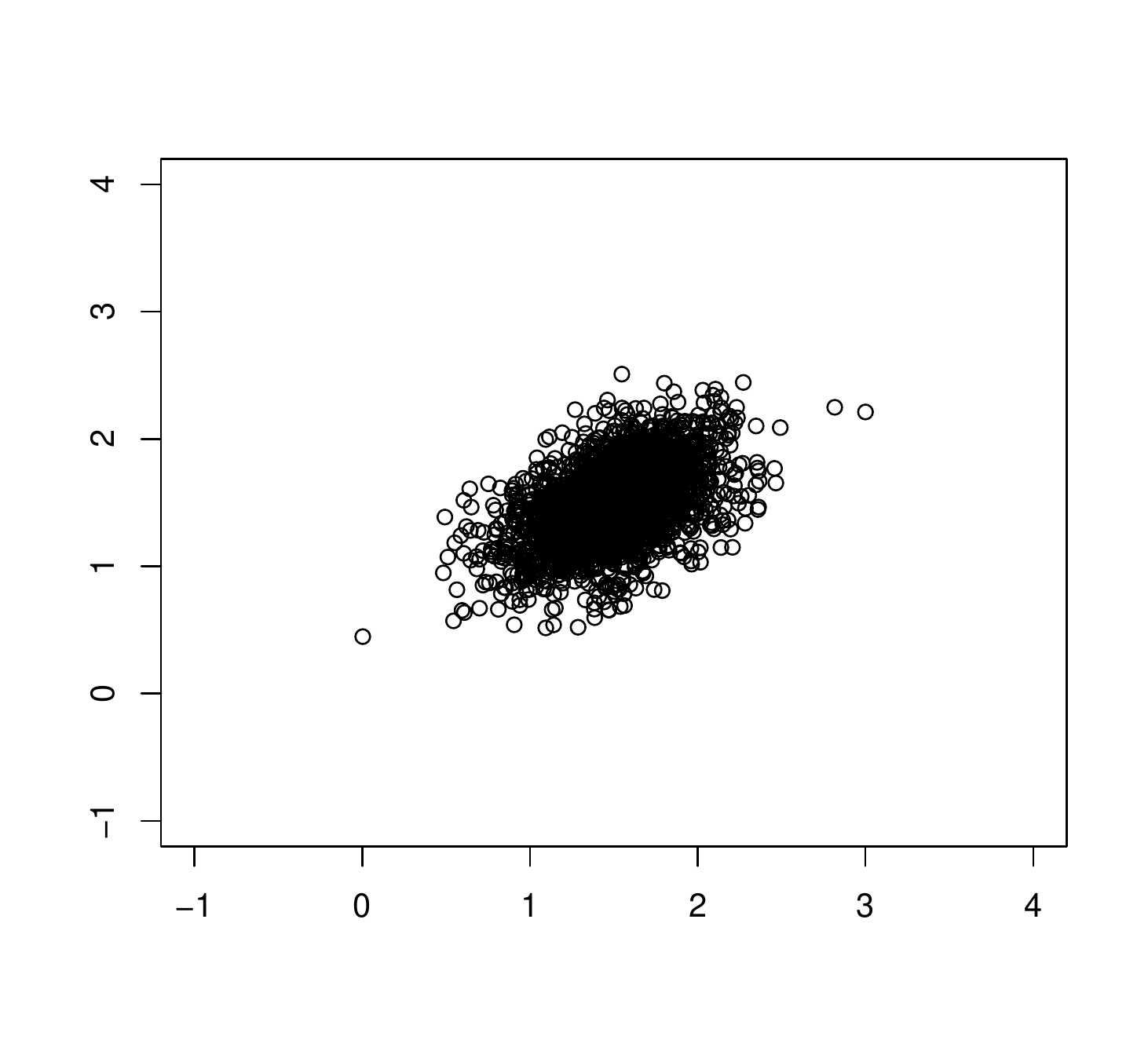}~
    \includegraphics[width=0.5\textwidth]{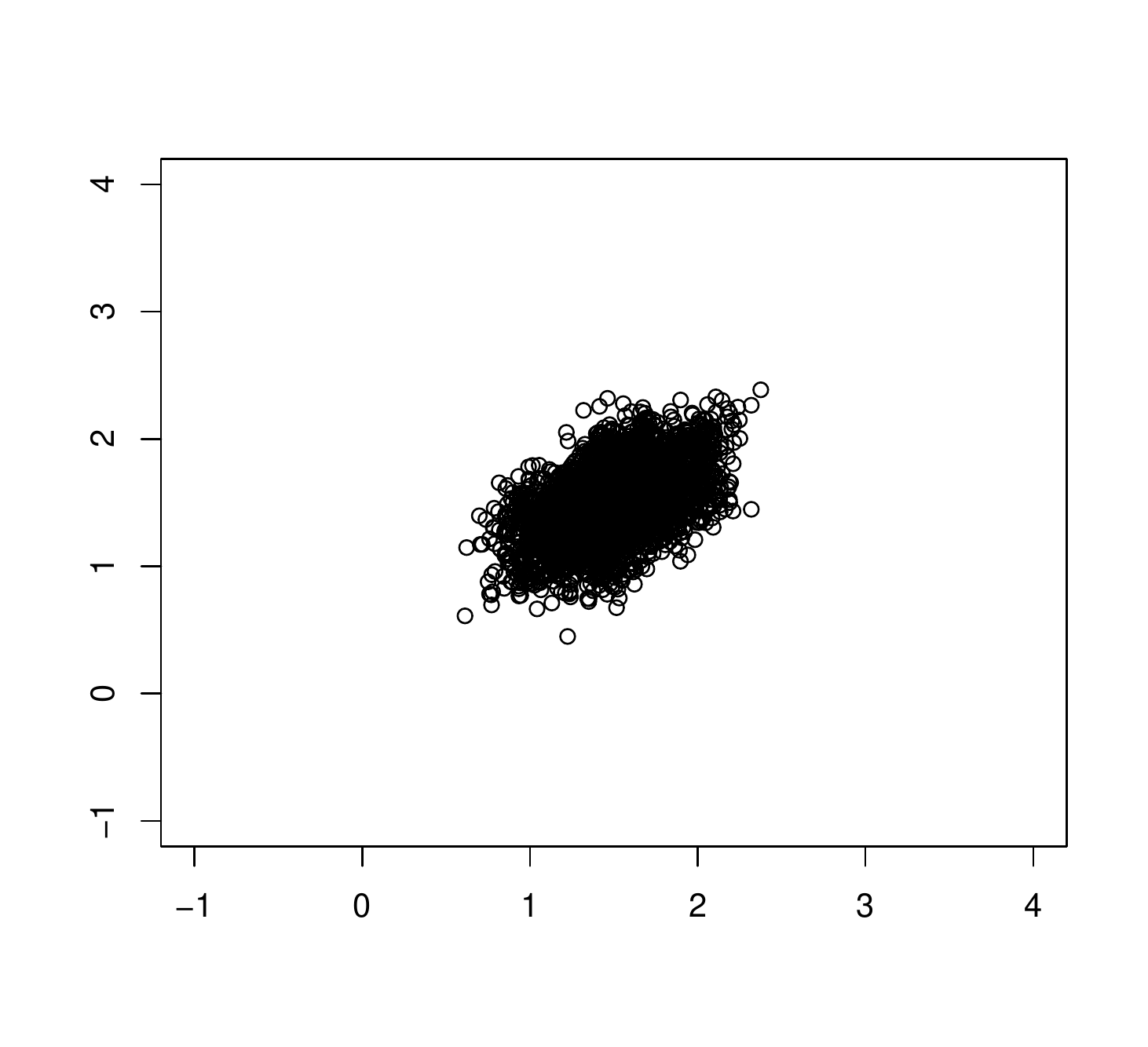}
    \caption{The result from an initial experiment for Algorithm \ref{alg:COM} with maximum iteration value $I=50$ and tropical distance, $\sigma=.1$.  (Left) Results taken from $3,000$ sampled points using (\ref{eq:gaus1}) in Metropolis-Hasting filtering. (Right) Results taken from $3,000$ sampled points using (\ref{eq:gaus2}) in the Metropolis-Hasting filtering.}
    \label{fig:THAR_NORM}
\end{figure}

Using~(\ref{eq:gaus2}) exhibits much less dispersion as compared with~(\ref{eq:gaus1}).  In either case we see that $B_r(x)_{tr}$ is well-defined.

For comparison between the classical HAR sampler with samples from $N(0,\sigma)$ and the tropical equivalent, Figure~\ref{fig:HAR_MOVES} shows a tracing of moves associated with Algorithms~\ref{alg:Gauss1} (left) and~\ref{alg:COM} (right).  The mechanism to make these moves is slightly different.  For Algorithm~\ref{alg:Gauss1}, Metropolis-Hastings filtering is not required so proposed moves are never rejected unless they are outside of a polytope (if a polytope is defined).  In Algorithm~\ref{alg:COM}, we first sample using a tropical HAR method with its own number of iterations, $I$, then assess whether to keep the proposed move using the filtering mechanism shown in Remark~\ref{rem:dist}.  While the output of Algorithm~\ref{alg:COM} is a sample of points, we can think of these sampled points as defining the number of moves used to decorrelate the starting point, $x_k$, and the next sampled point, $x_{I}$ (the final point reached).  

\begin{figure}[H]
    \centering
    \includegraphics[width=0.49\textwidth]{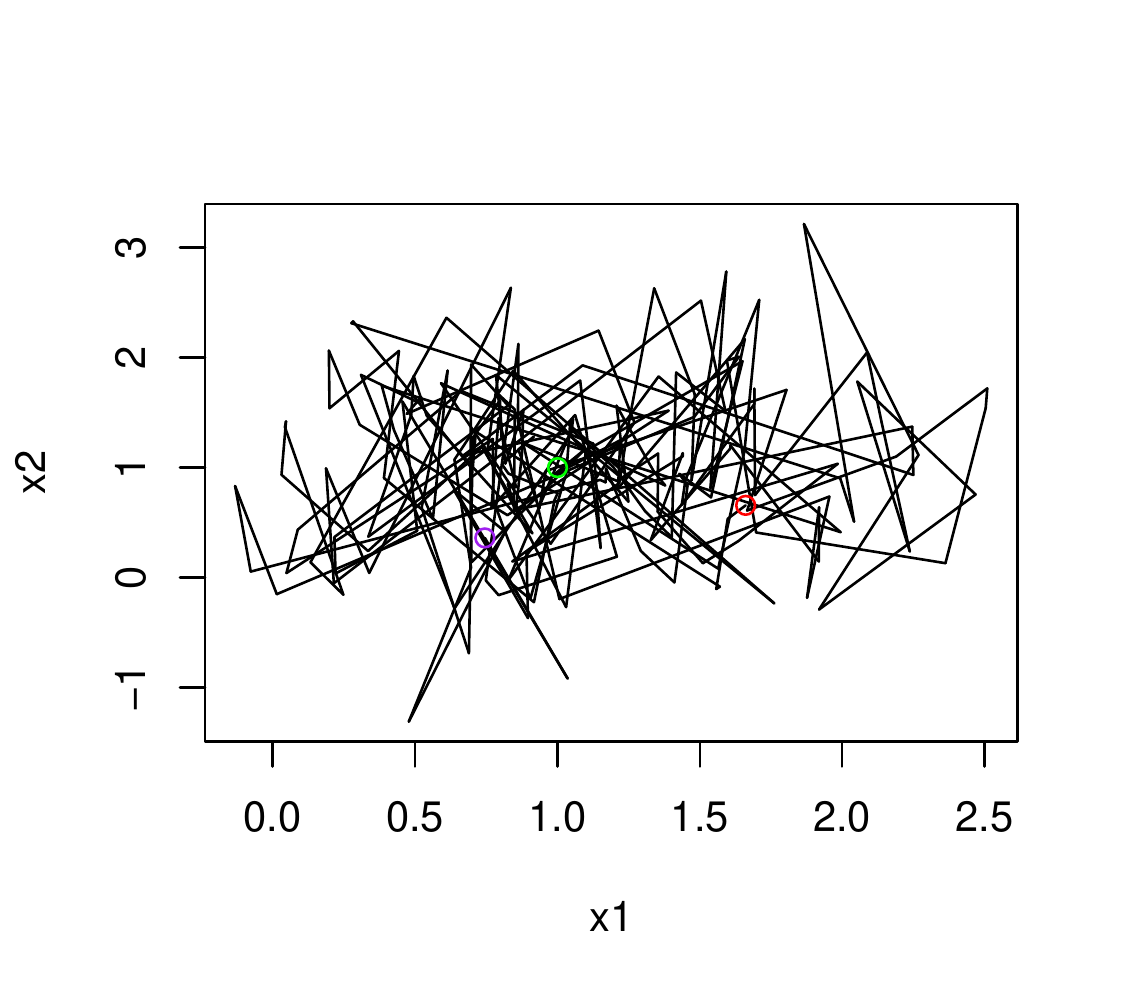}
    \includegraphics[width=0.49\textwidth]{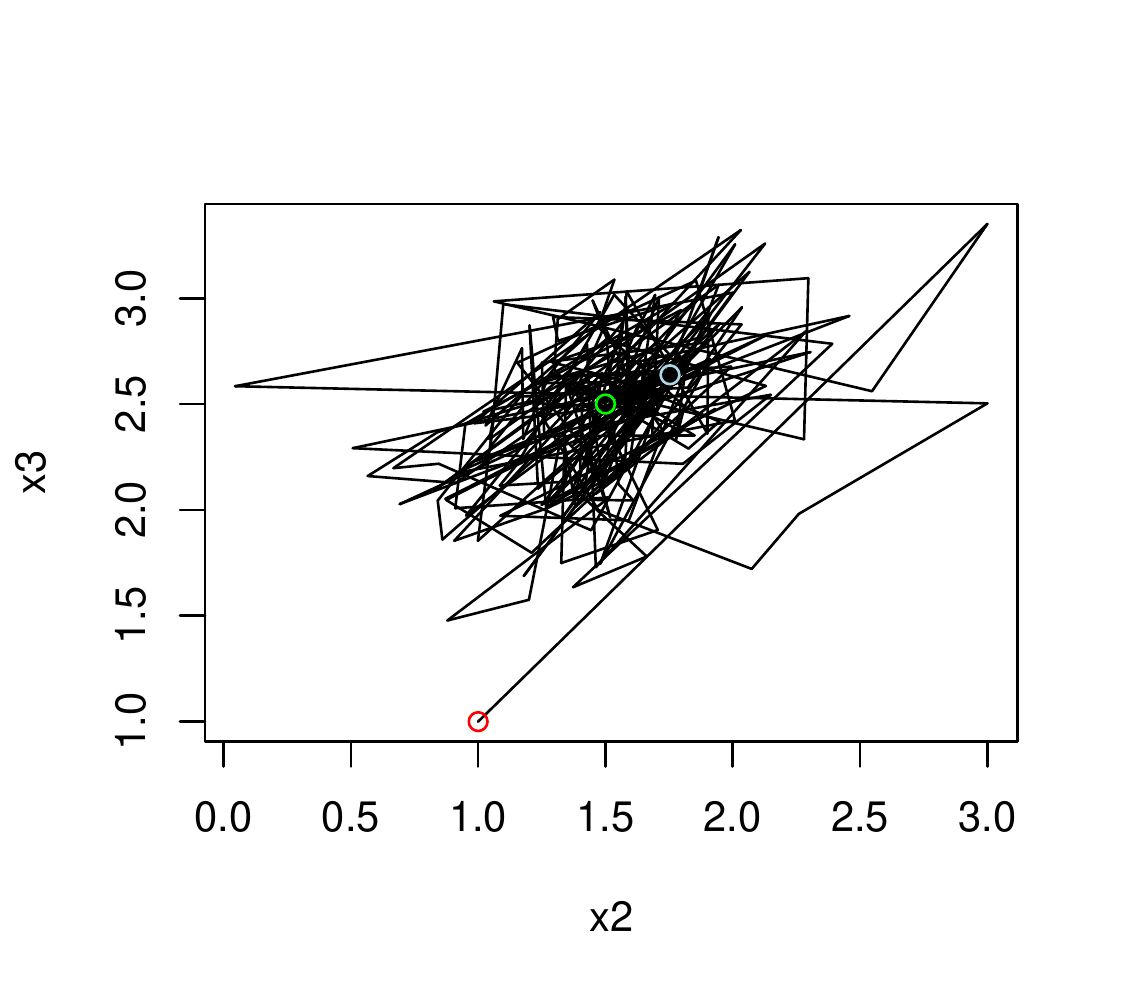}
    \caption{HAR movement to obtain one sample using $I=200$ iterations.  (left) Algorithm~\ref{alg:Gauss1} results.  The {\color{red}red} point is the starting point. The {\color{green} green} point is the center of mass, and the {\color{violet} purple} point is the resultant sampled point.  (right) Algorithm~\ref{alg:COM} moves when sampling from a polytrope. The {\color{red}red} point is the starting point. The {\color{cyan}blue} point is the sampled point. The center of mass point, $\mu$, is the {\color{green}green} point.}
    \label{fig:HAR_MOVES}
\end{figure}

The left figure shows how the algorithm traverses around the center of mass point ({\color{green}green} point) which is what we would expect for sample points taken from a normal distribution.  The sequence of moves for the figure on the right above shows the movements determined by the Metropolis-Hastings filtering in Algorithm~\ref{alg:COM} around $\mu$.

\section{Computational Experiments for Sampling from a Tropical Convex Hull}\label{EXP}

\subsection{Sampling from Tropical Polytopes over $\mathbb{R}^4/\mathbb{R}{\bf 1}$}

In this experiment we consider a tropical polytope with vertices
\[
\begin{array}{ccc}
v_1&=&(0, 0, 0, 0)\\
v_2&=&(0, 1, 3, 1)\\
v_3&=&(0, 1, 2, 5)\\
v_4&=&(0, 2, 5, 10).\\
\end{array}
\]
This tropical polytope is shown in Figure \ref{fig:4dimtropicalpolytope} which is drawn by M.~Joswig using the software {\tt polymake} \cite{polymake:2000}.
\begin{figure}[H]
    \centering
    \includegraphics[width=0.7\textwidth]{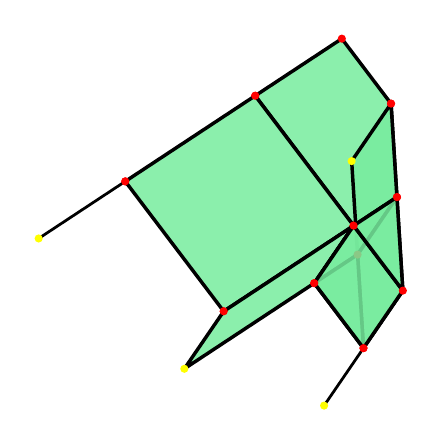}
    \caption{Tropical polytope in $\mathbb{R}^4/\mathbb{R}{\bf 1}$ drawn using {\tt polymake}.  Yellow points represent vertices of the tropical polytope. }
    \label{fig:4dimtropicalpolytope}
\end{figure}

This experiment compares the use of the Algorithms~\ref{alg:HAR_vert3} and~\ref{alg:HAR_extrapolation2} on the tropical polytope defined by the vertices above and shown in Figure~\ref{fig:4dimtropicalpolytope}.  In both cases we obtain a sample size of $5,000$ with $I=500$.  In addition we incorporate a ``burn-in,'' $b=1,000$.  The point of $b$ is to allow the algorithm time to move away from areas of the polytope that might lead to biased sampling due to the structure of the polytope itself.  For Algorithm~\ref{alg:HAR_vert3} we have a $d=4$. So given a tropical line segment, $\Gamma_{x_0,v}$, with end points $x_0$ and $v$, and breakpoints, $b_{x_0}$ and $b_v$, $d_{tr}(b_{x_0},u')=d_{tr}(b_{v},v')=4$ where $u'$ and $v'$ are the extended end points.


\begin{figure}[H]
    \centering
    \includegraphics[width=0.3\textwidth]{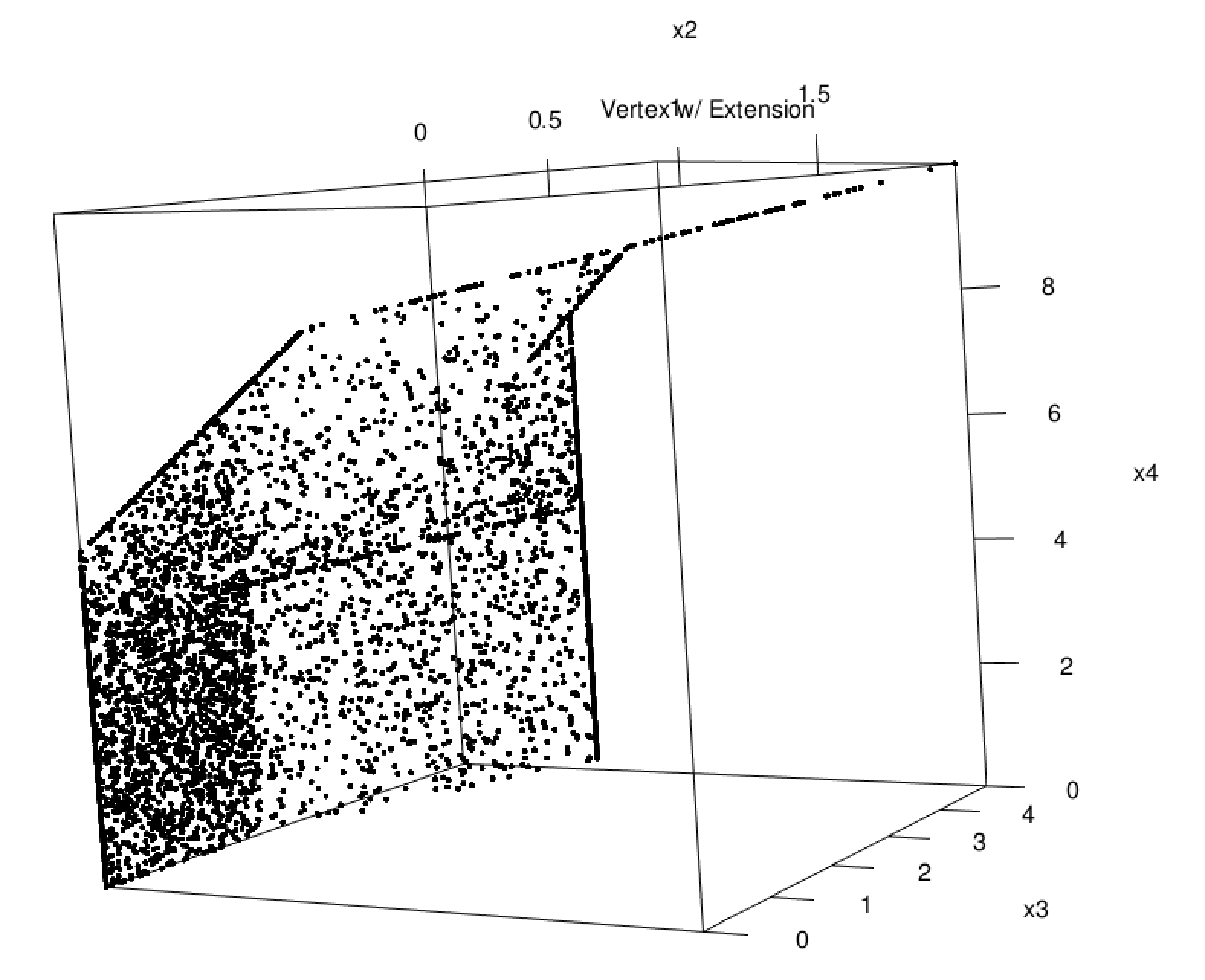}~
    \includegraphics[width=0.3\textwidth]{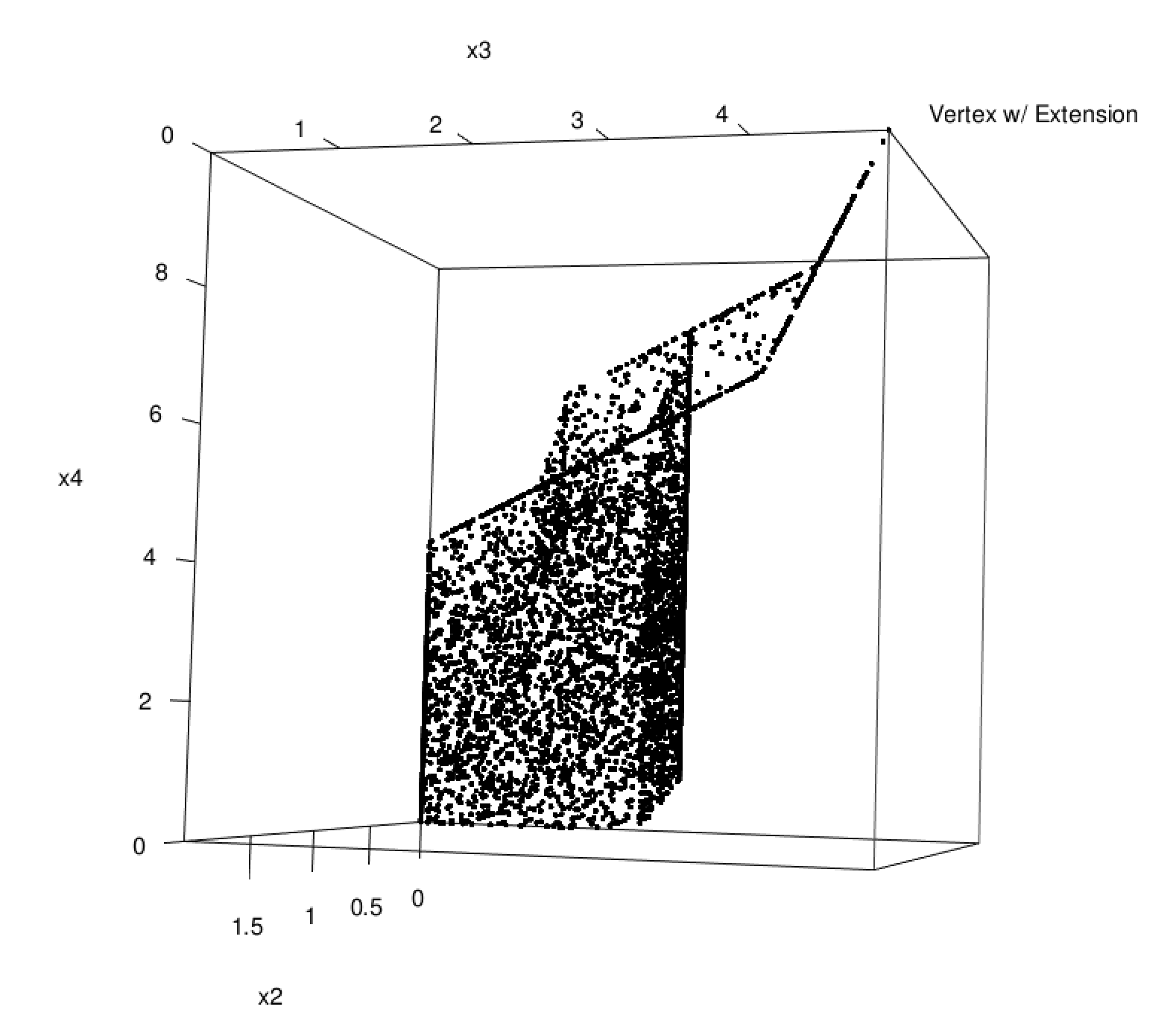}
    \includegraphics[width=0.3\textwidth]{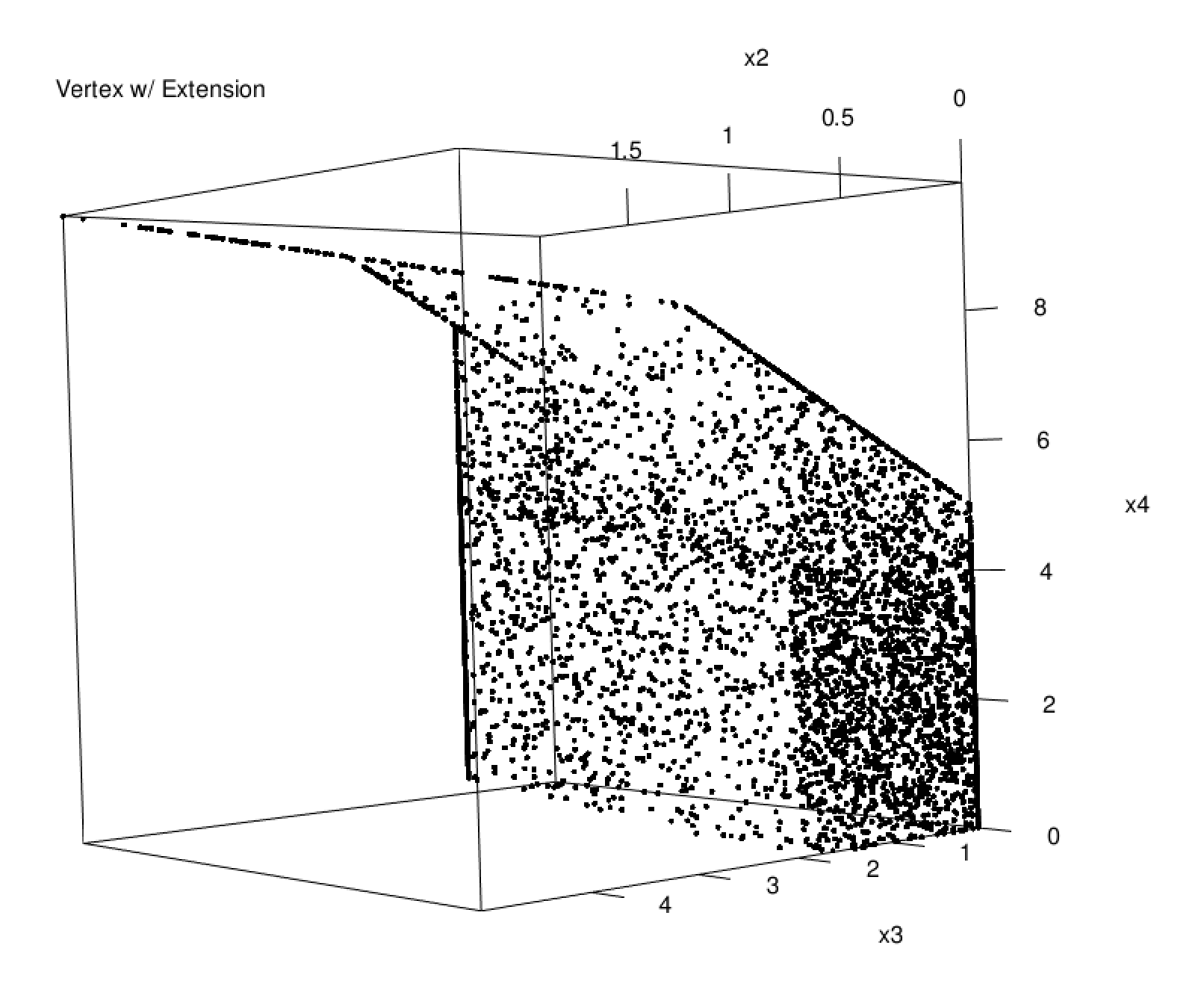}
    \includegraphics[width=0.3\textwidth]{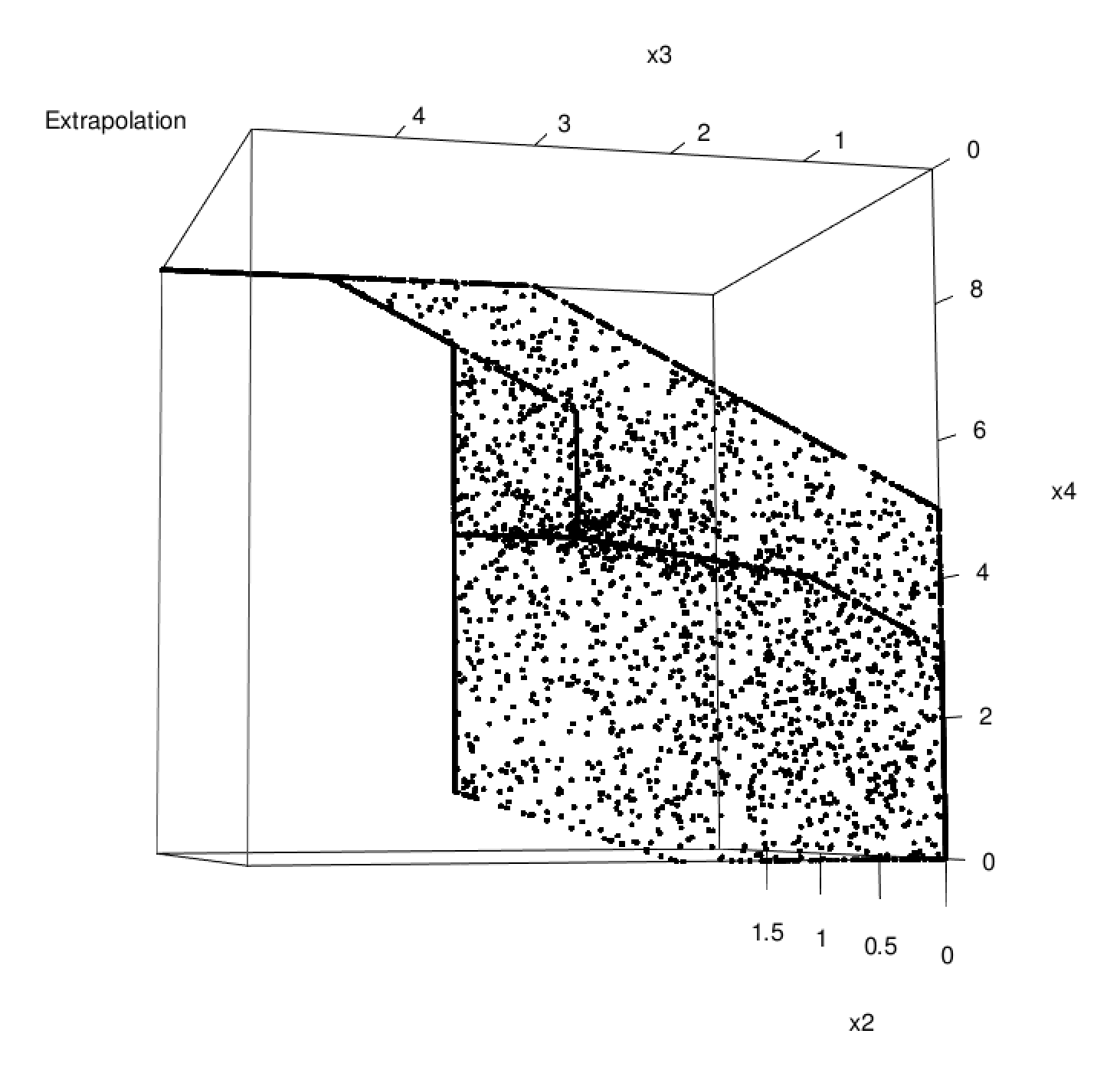}~
    \includegraphics[width=0.3\textwidth]{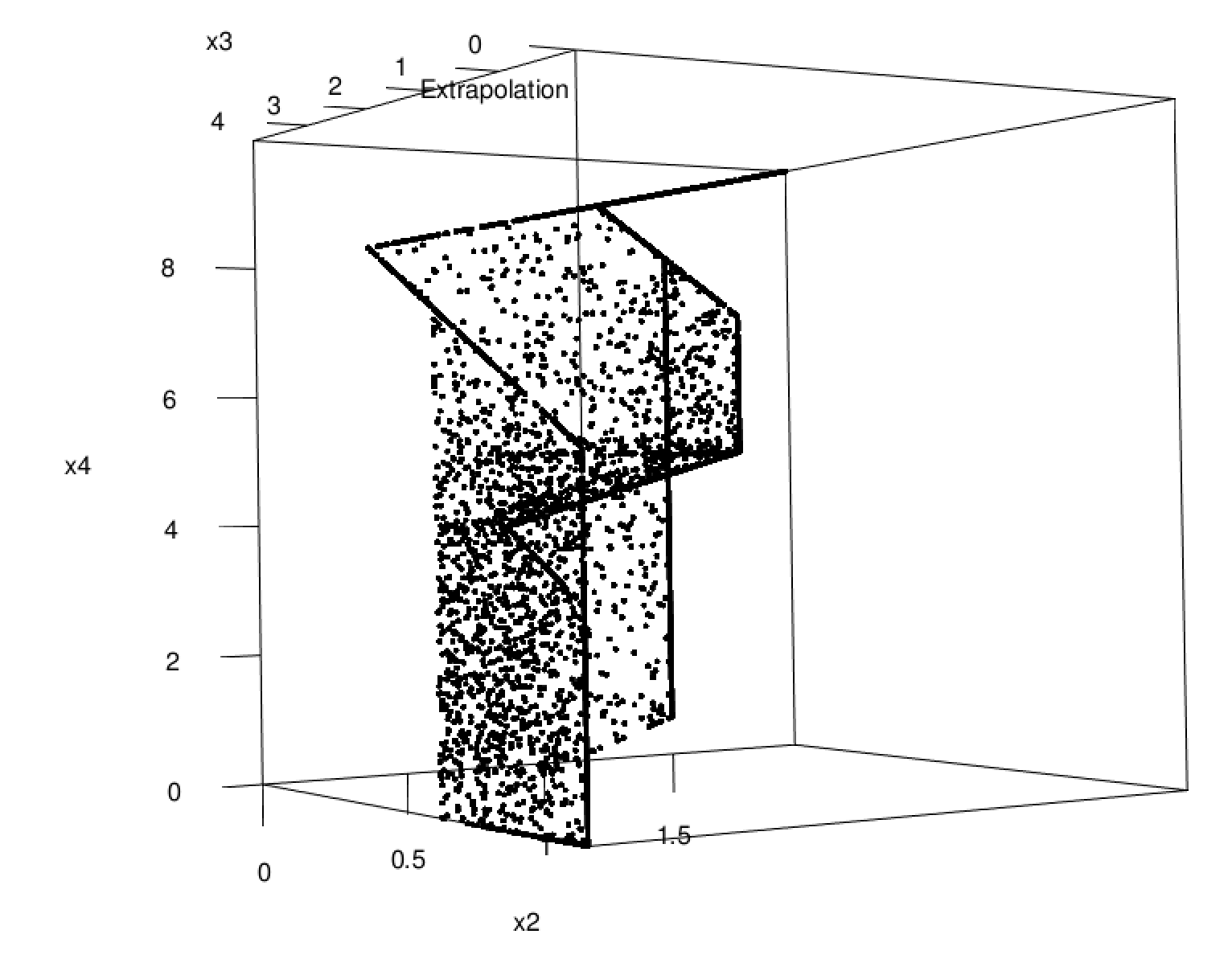}
    \includegraphics[width=0.3\textwidth]{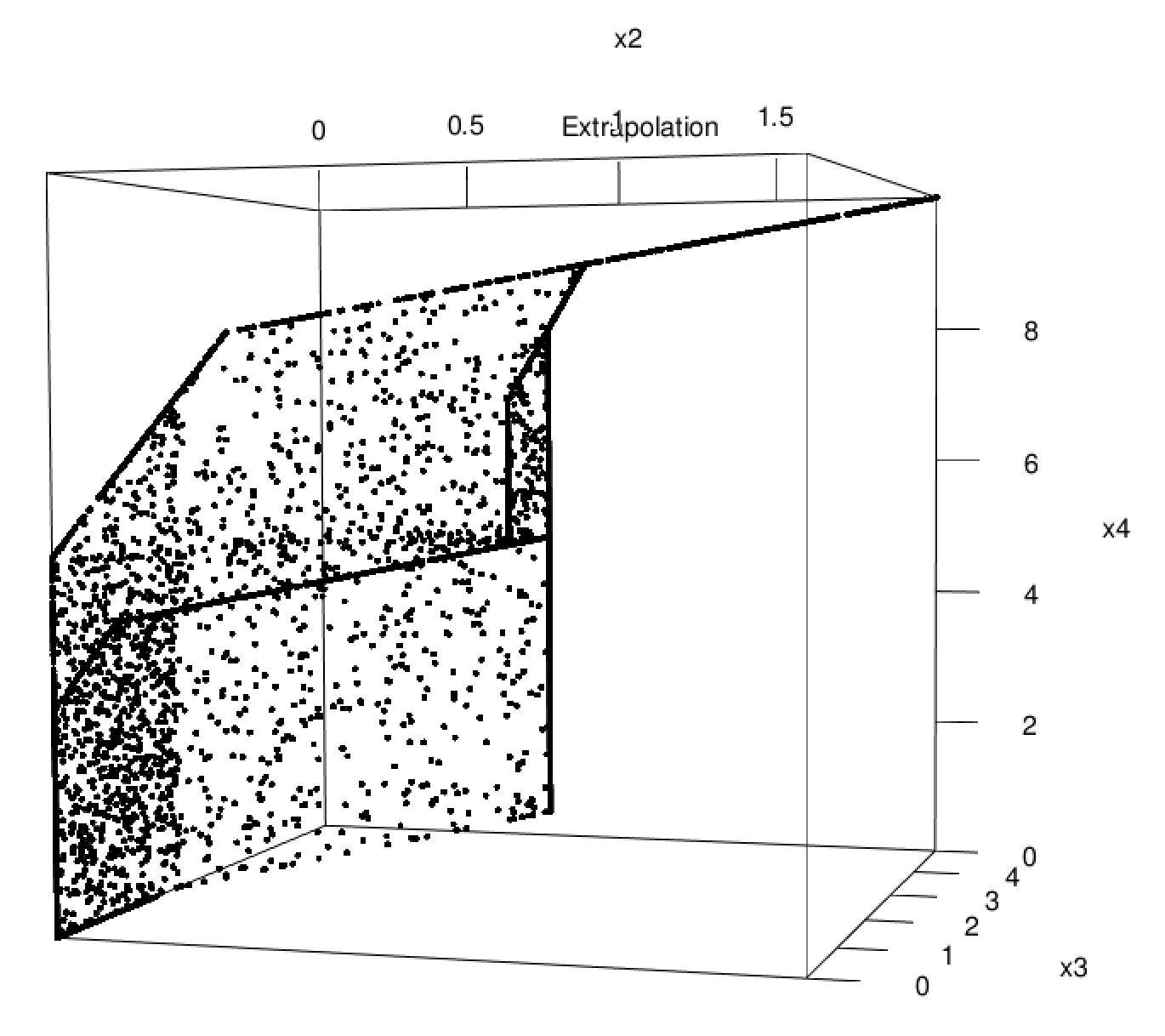}
    \caption{Results using Algorithm~\ref{alg:HAR_vert3} (top row) and Algorithm~\ref{alg:HAR_extrapolation2} (bottom row) to sample 5,000 points from a polytope, $\mathcal{P}\in \mathbb{R}^4/\mathbb{R}\textbf{1}$.}
    \label{fig:ext_5000}
\end{figure}


Both algorithms provide definition to the polytope, $\mathcal{P}$, however, Algorithm~\ref{alg:HAR_extrapolation2} sampling biases towards the edges of $\mathcal{P}$.  By contrast, Algorithm~\ref{alg:HAR_vert3} more reliably samples throughout, $\mathcal{P}$.  Nonetheless, there are edges where Algorithm~\ref{alg:HAR_vert3} exhibits bias especially along some of the one-dimensional faces of $\mathcal{P}$.

\subsubsection{Sampling from Tropical Polytopes with a Given Distribution}

Considering the tropical polytope defined by the vertices in the previous section, we now experiment with Algorithm~\ref{alg:COM}.  Figure~\ref{fig:com_2000} shows results of increasing values of $\sigma_{tr}$ with 2,000 sampled points each with a center of mass point $\mu$.  This is tantamount to progressively increasing the radius, $r$, of the tropical ball $B_r(x)_{tr}$.

\begin{figure}[H]
    \centering
    \includegraphics[width=0.4\textwidth]{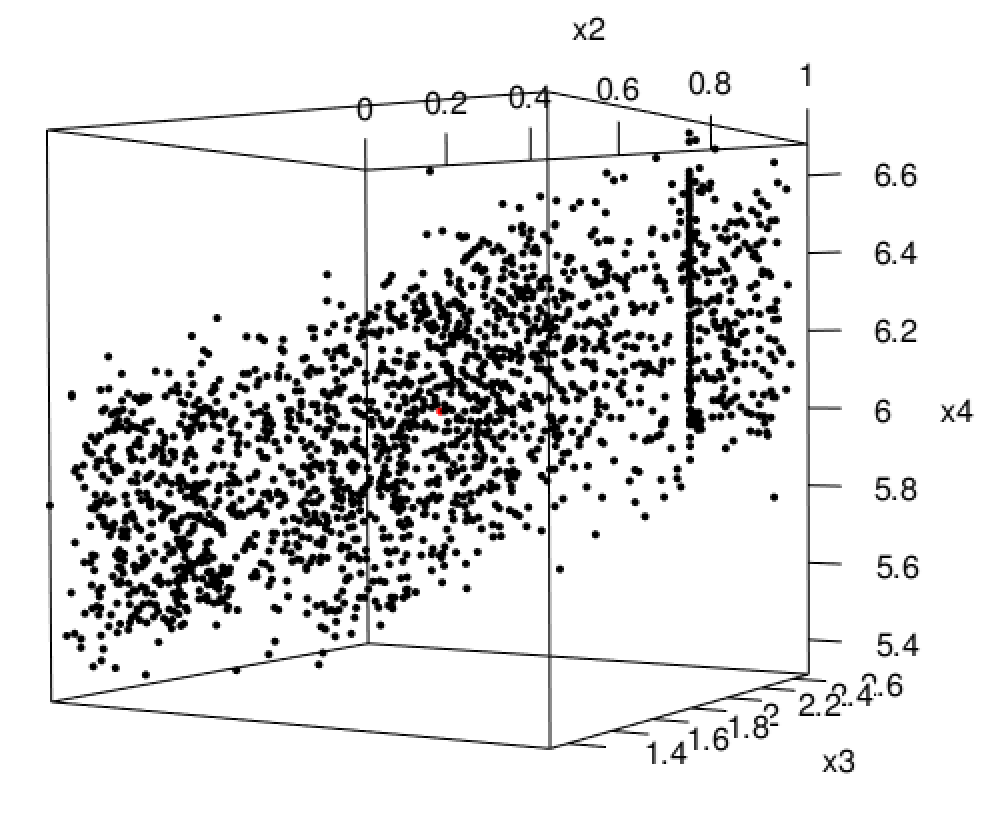}~
    \includegraphics[width=0.4\textwidth]{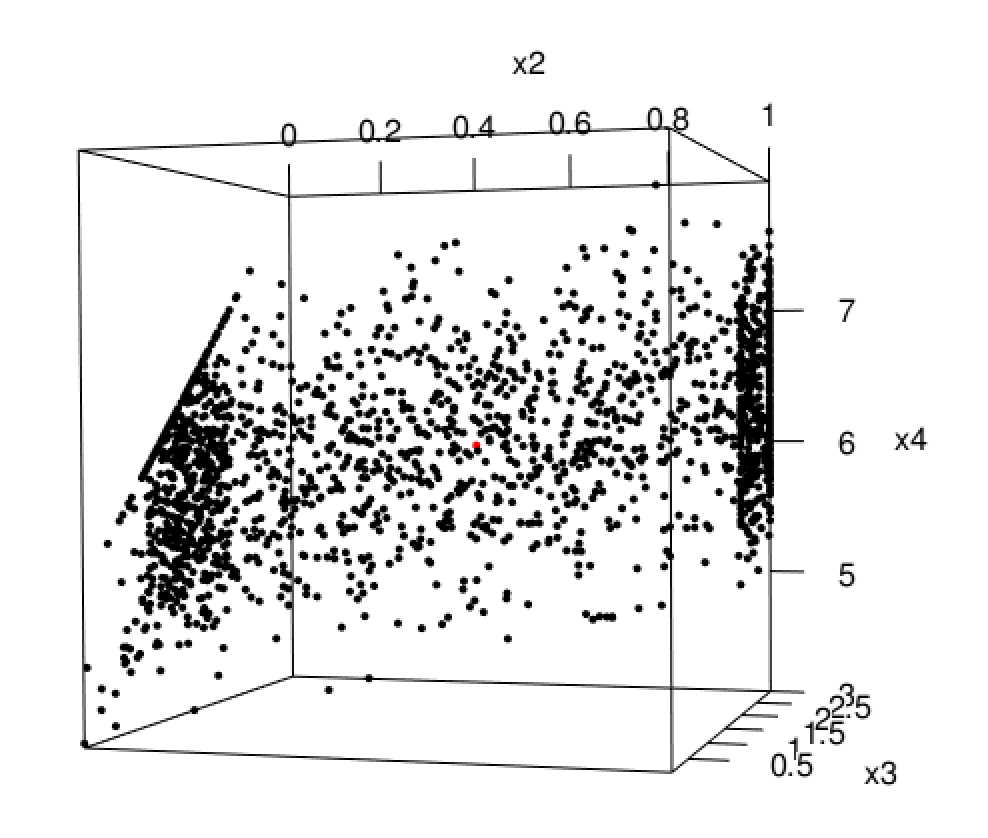}
    \includegraphics[width=0.4\textwidth]{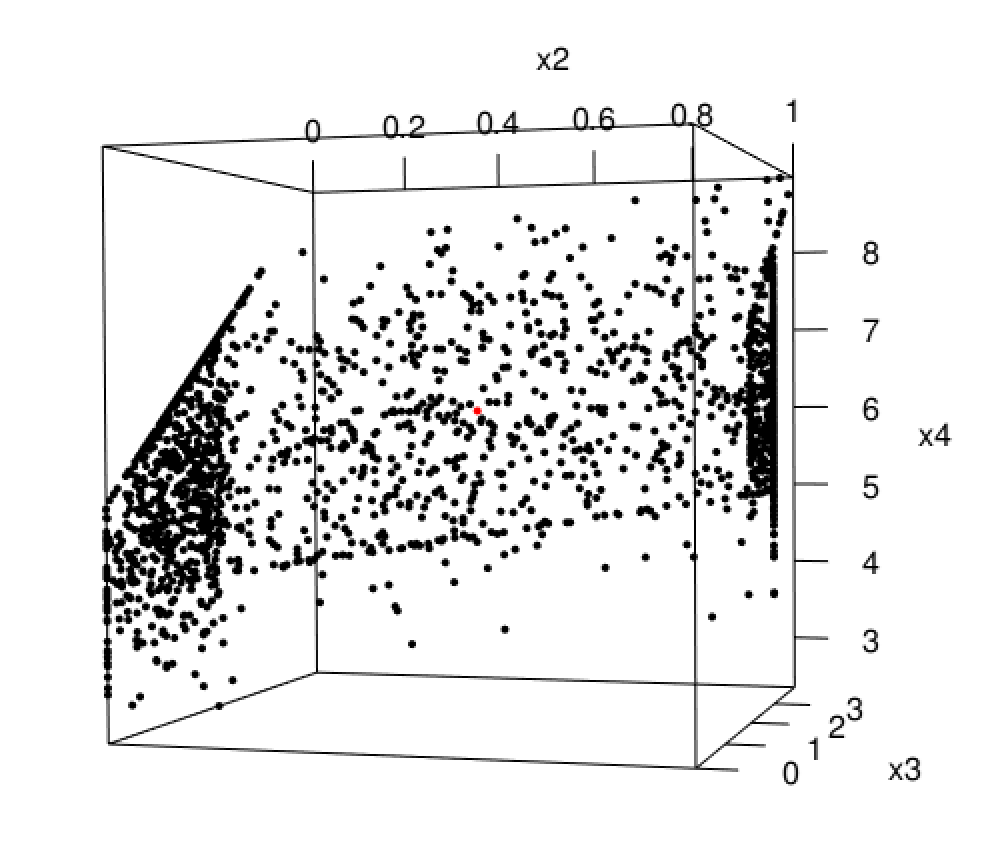}
    \includegraphics[width=0.4\textwidth]{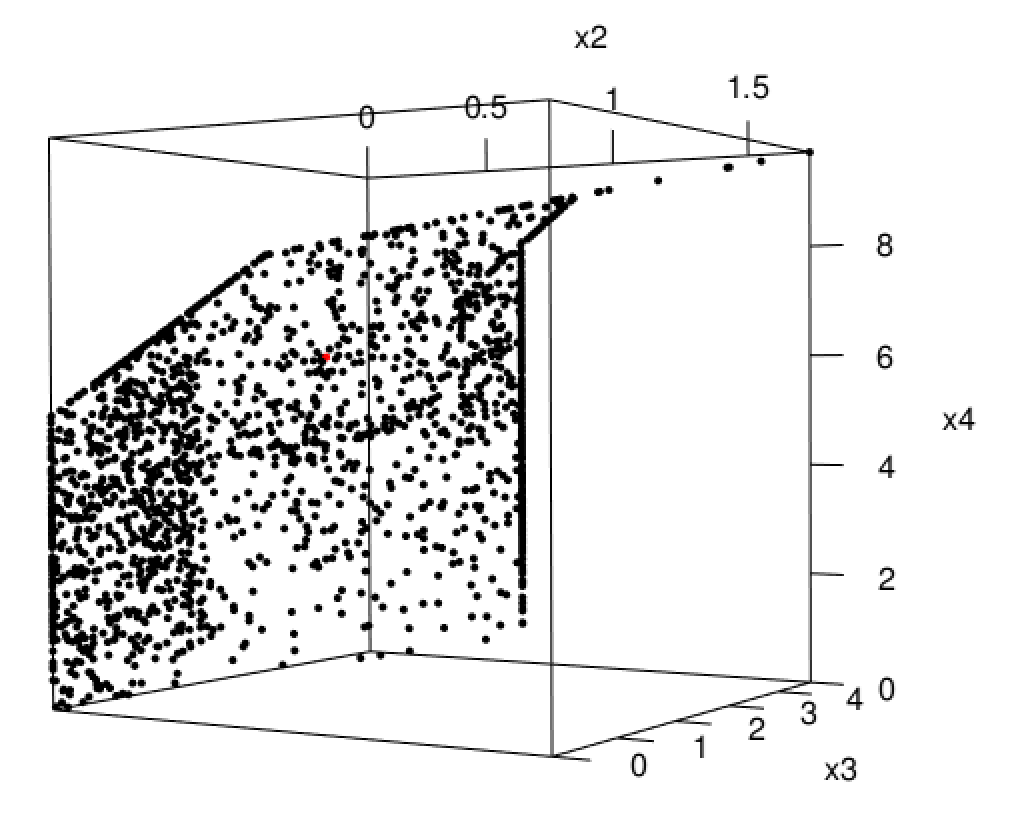}~
    \caption{Results using Algorithm~\ref{alg:COM} to sample 2,000 points from a polytope, $\mathcal{P}\in \mathbb{R}^4/\mathbb{R}\textbf{1}$ with $\mu=(0, 0.5, 2, 6)$ and  $\sigma_{tr}=.05$ (top left), $\sigma_{tr}=.5$ (top right), $\sigma_{tr}=2$ (bottom left), and $\sigma_{tr}=30$ (bottom right). In all cases the {\color{red}red} point is $\mu$.}
    \label{fig:com_2000}
\end{figure}

In the example above, $\mu=(0,0.5,2,6)$ which happens to be in a full-dimensional portion of the tropical polytope, $\mathcal{P}$.  As $\sigma_{tr}$ increases, more points are sampled further away from $\mu$, giving visual definition to $\mathcal{P}$.  To get full definition of $\mathcal{P}$, $\sigma_{tr}$ must be increased well beyond the maximum tropical distance from any two points in $\mathcal{P}$ (in this case, $\{\max\limits_{x,x'}\; d_{tr}(x,x')\;|\;x,x'\in\mathcal{P}\}=10$).

\section{Sampling from the Space of Ultrametrics}\label{sec:ultra}

\subsection{Basics of Ultrametrics}

Suppose we have $[m] := \{1, \ldots , m\}$ and let $u: [m] \times [m] \to \RR$ be a metric over $[m]$, that is, $u$ is a map from $[m]\times [m]$ to $\RR$ such that
\begin{eqnarray}\nonumber
u(i, j) = u(j, i) & \mbox{for all } i, j \in [m]\\\nonumber
u(i, j) = 0 & \mbox{if and only if } i = j\\\nonumber
u(i, j) \leq u(i, k) + u(j, k) & \mbox{for all }i, j, k \in [m].
\end{eqnarray}

Suppose $u$ is a metric on $[m]$.  Then if $u$ satisfies 
\begin{eqnarray}
\max\{u(i, j), u(i, k), u(j, k)\} \mbox{ and is achieved at least twice,}
\end{eqnarray}
then $u$ is called an {\em ultrametric}.  

\begin{example}
Suppose $m = 3$.  Then a metric $u$ on $[m]$ such that
\[
u(1, 2) = 2, \, u(1, 3) = 2, u(2, 3) = 1,
\]
is an ultrametric.
\end{example}

A phylogenetic tree is a weighted tree whose internal nodes do not have labels and whose external nodes, i.e., leaves, have labels.  We consider a rooted phylogenetic tree with a given leaf label set $[m]$.  
\begin{definition}
Suppose we have a rooted phylogenetic tree $T$ with a leaf label set $[m]$.  If a distance from its root to each leaf $i \in [m]$ is the same distance for all $i \in [m]$, then we call $T$ an {\em equidistant tree}.
\end{definition}

\begin{example}
The phylogenetic tree shown in Figure \ref{fig:eqEx} is an equidistant tree with a leaf label set $[5]$ and its pairwise distances are 
\[
u = (4, 4, 4, 4, 2, 2, 2, 1.6, 1.6, 0.6)
\]
which is an ultrametric.
\begin{figure}[h]
    \centering
    \includegraphics[width=0.25\textwidth]{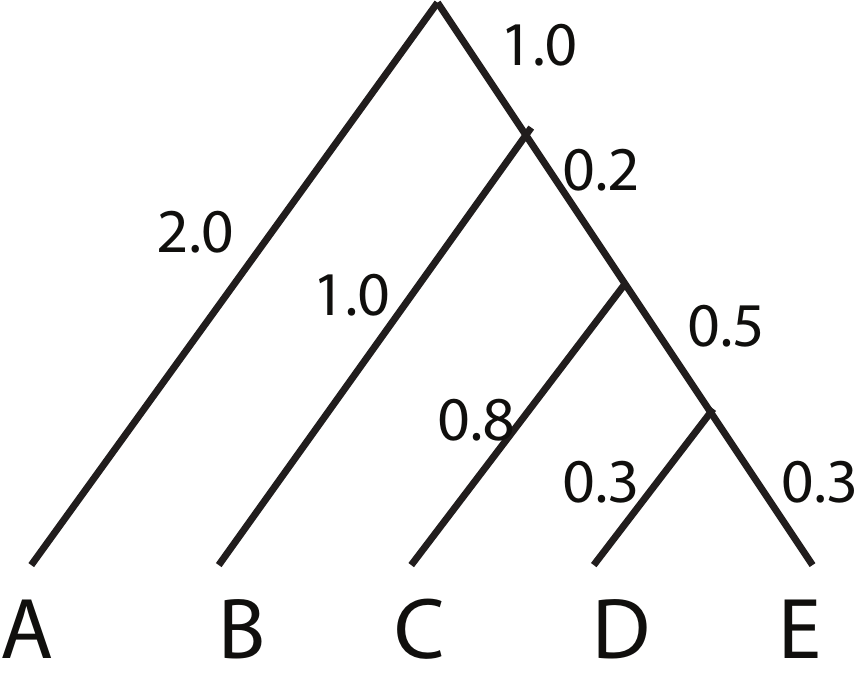}
    \caption{Example of an equidistant tree with a leaf label set $[5]$.}
    \label{fig:eqEx}
\end{figure}
\end{example}

\begin{theorem}[\cite{Buneman}]\label{thm:3pt}
Suppose we have an equidistant tree $T$ with a leaf label set $[m]$ and suppose $u(i, j)$ for all $i, j \in [m]$ is a distance from a leaf $i$ to a leaf $j$.  Then, $u$ is an ultrametric if and only if $T$ is an equidistant tree. 
\end{theorem}

Using Theorem \ref{thm:3pt}, if we consider all possible equidistant trees, then we can consider the space of ultrametrics as the space of phylogenetic trees on $[m]$.

We consider the linear subspace $L_m \subseteq \mathbb{R}^e$, where $e = \binom{m}{2}$,
defined by the linear equations
\begin{equation}
\label{eq:trop_eq}
u_{ij} - u_{ik} + u_{jk}=0
\end{equation} 
for $1\leq i < j < k \leq m$ in $u$.  
The (max-plus) tropicalization of the linear space $L_m$, denoted by $\Trop(L_N) \subseteq \mathbb{R}^{e}/\mathbb{R}{\bf 1}$, is the {\em tropical linear space} defined by ultrametrics $u_{ij} \oplus u_{ik} \oplus u_{jk}$, i.e., $\max\big(u_{ij},\, u_{ik},\, u_{jk} \big)$ is obtained at least twice for all triples $i,\,j,\,k \in [m]$.  Then we have the following theorem:
\begin{theorem}[\cite{AK}]
\label{th:treespace}
The image of $\mathcal{U}_m$ in the tropical projective torus $\mathbb{R}^{e}/\mathbb{R}{\bf 1}$ coincides with $\Trop(L_m)$.  
\end{theorem}
Note that a tropical linear space is tropically convex.  Therefore, by Theorem \ref{th:treespace}, the space of ultrametrics is tropically convex.  Thus, in this section, we apply our HAR sampler to sample an ultrametric (equidistant tree) with $[m]$, randomly.  

\subsection{HAR Algorithm on Space of Ultrametrics}

Now we consider applying a HAR algorithm to the space of ultrametrics.  
Suppose we have a set of leaves $[m] := \{1, \ldots , m\}$ and we consider equidistant trees with the same height $h > 0$, that is, rooted phylogenetic trees whose distances from their root to each leaf are $h$, with leaf labels $[m]$.
We consider the space of ultrametrics~ $\Tn \subset \T$ as the space of equidistant trees with $[m]$.  
For tropical HAR algorithm to sample a point from the space of ultrametrics we use UPGMA (unweighted pair group method with arithmetic mean) \cite{sokal58}.
UPGMA is an algorithm to project any $m\times m$ symmetric matrix with its diagonal equal zero onto $\Tn$.  

\begin{algorithm}
\caption{Sampling via HAR algorithm from $\Tn$}\label{alg:HAR3}
\begin{algorithmic}
\State {\bf Input:} Initial point $x_0 \in \Tn$ and maximum iteration $I \geq 1$.
\State {\bf Output:} A random point $x \in \Tn$.
\State Set $k = 0$.
\For{$k= 0, \ldots , I-1$,}
\State Generate a random direction $D^k+(D^k_1, \ldots, D^k_e)$ such that $D^k_i$ is sampled uniformly from $[0, \max(x_0)]$. 
\State Use UPGMA to project a point $y:= x_k + \lambda \cdot D^k$, where $\lambda > 0$, onto $\Tn$. Let $\pi_{\Tn}(y)$ be the projection of $y$ onto $\Tn$. 
\State Generate a random point $x_{k+1}$ from a tropical line segment $\Gamma^k_{x_k, \pi_{\Tn}(y)}$ using Algorithm \ref{eq:troline}.
\EndFor \\
\Return $x := x_{I}$.
\end{algorithmic}
\end{algorithm}

\begin{proposition}
With Algorithm \ref{alg:HAR3}, the time complexity to sample an ultrametric from $\Tn$ via HAR sampling with $I$ iterations is $O(Im^2)$.
\end{proposition}
\begin{proof}
The time complexity of UPGMA is $O(m^2)$. Since $e = \binom{m}{2}$, the time complexity with with $I$ iterations via Algorithm \ref{alg:HAR3} is $O(Im^2)$.
\end{proof}

\begin{example}
We set $m = 4$ which yields 15 possible unique tree topologies.  For the HAR algorithm, we used $I = 30$ and sampled $10,000$ observations from the space of ultrametrics using Algorithm \ref{alg:HAR3}.  We used $(0.1, 1, 0.67, 1, 0.67, 1) \in \mathcal{U}_4$ as an initial point.  The result is shown in the left of Figure \ref{fig:res2}.  Each bin in the histogram in Figure \ref{fig:res2} represents each tree topology of the 15 tree topologies for equidistant trees.  
To compare with our simulation, we conducted sampling using the {\tt rcoal()} function in the {\tt phangorn} package in {\tt R} \cite{phangorn}. 
Then we increased the sample size to $10,000$ with the {\tt rcoal()} function and the result is shown in the right of Figure \ref{fig:res2}.
\begin{figure}[H]
    \centering
    \includegraphics[width=0.44\textwidth]{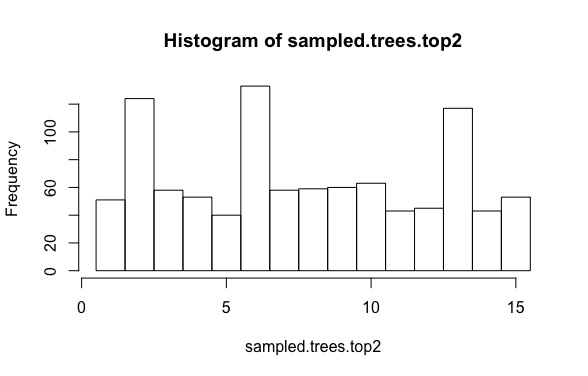}
     \includegraphics[width=0.44\textwidth]{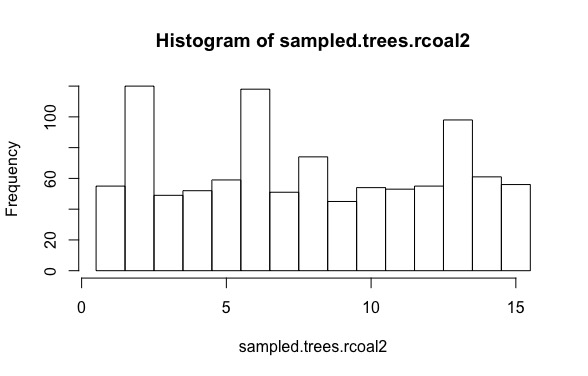}
    \caption{(Left) The result from an initial experiment with Algorithm \ref{alg:HAR3} with $10,000$ sampled points. (Right) The result from an experiment using the {\tt rcoal} function in the {\tt phangorn} package with $10,000$ sampled points.  Each bin represents each unique tree topology of equidistant trees with $m = 4$ leaves.}
    \label{fig:res2}\label{fig:res4}
\end{figure}
\end{example}

\section{Conclusion}

In this paper we introduce novel HAR MCMC methods to sample from a tropical polytope and show that these methods perform well for a variety of generic tropical polytopes.  In addition, we show that using these algorithms we can sample from a tropical polytope uniformly.  We extend these HAR methodologies to sample according to a distribution of choice with  Metropolis-Hastings filtering.

Even though we show that while Algorithm~\ref{alg:HAR_vert3} samples well throughout given tropical polytopes of varying dimensions via computational experiments, a transition probability of a Markov chain is not symmetric while a transition probability of a Markov chain in Algorithms \ref{alg:HAR_extrapolation} and \ref{alg:HAR_extrapolation2} is symmetric. It seems that the key element of sampling from the uniform distribution over a given tropical polytope via HAR sampler is how to draw a random tropical line segment from a point in the boundary of the tropical polytope to another point on the boundary.  Thus, even though a transition probability of a Markov chain in the vertex HAR with extended tropical line segments is not symmetric, it is worth investigating how to make these transition probabilities symmetric.  It is also worth noting that if we solely rely on rejection sampling,  Algorithm~\ref{alg:HAR_vert3} is not very efficient on sampling from a given tropical polytope.  Therefore, we project end points of an extended line segment onto the tropical polytope and then we project these projections back onto the extension of the line segment. If we do, we can reduce the computational time of Algorithm~\ref{alg:HAR_vert3}.  


The next step is to apply a proposed HAR sampling method to problems in polyhedral geometry and statistics.  One of applications of classical HAR sampler over an Euclidean space is to estimate the volume of a classical polytope, which is the convex hull of finitely many vertices.  In similar fashion we are also interested in the application of our HAR sampler using the tropical metric to the {\em volume} of a tropical polytope over the tropical semiring using the max-plus algebra \cite{Joswig}.  It is well-known that estimating the volume of a tropical polytope is very hard \cite{Gaubert}. Our idea is to apply our proposed HAR sampling method from a tropical polytope which is an
analogue to the classical volume estimation methods developed by Cousins and Vempala~\citep{CV}.   Our initial computational experiments suggest that our HAR sampler with the distribution proportional to equation \eqref{eq:gaus2} from a tropical polytope works very well to estimate its volume.  

We can also consider the use of tropical HAR methods in the field of causal inference using the max-linear Bayesian network for extreme value statistics.  With our HAR sampling method we might be able to generalize the Wang-Stoev conditional sampling algorithm for parameterized models in extreme spatial statistics and time-series to learn a directed acyclic graph (DAG) for causal inference~\citep{Tran}.

\section*{Acknowledgement}

The authors thank Profs.~Michael Joswig and Ngoc Tran for useful conversations and discussions. RY and DB are partially supported from NSF DMS 1916037.
KM is partially supported by JSPS KAKENHI Grant Numbers	JP18K11485, JP22K19816, JP22H02364.

\bibliographystyle{plain}  
\bibliography{refs}  

\end{document}